\documentclass{article}       
\usepackage{graphicx}
\usepackage[latin1]{inputenc}
\usepackage[english]{babel}
\usepackage{graphicx}
\usepackage{amsmath,amsfonts,amssymb,amsthm}
\usepackage{bbm,wasysym,stmaryrd}
\usepackage{subfigure}
\usepackage{color}
\usepackage{ifthen}
\usepackage{hyperref}
\usepackage[normalem]{ulem}

\newboolean{DisplayBigFigs}
\setboolean{DisplayBigFigs}{true}

\newcommand{\R}{\mathbb{R}}
\newcommand{\N}{\mathbb{N}}

\newcommand{\der}[2]{ \frac{\text{d} #1}{\text{d} #2} }  
\newcommand{\Matrix}[4]{\left (\begin{array}{cc} #1 & #2 \\ #3 & #4 \end{array}\right)}
\newcommand{\Vector}[2]{\left (\begin{array}{cc} #1 \\ #2  \end{array}\right)}

\theoremstyle{plain}
\newtheorem{theorem}{Theorem}
\newtheorem{prop}{Proposition}

\newtheorem{rem}{Remark}

\newcommand{\Exp}[1]{\mathbb{E}\left [ #1 \right]}

\newcommand{\eps}{\varepsilon}
 
\newcommand{\M}{\mathcal{M}}

\definecolor{orange}{rgb}{1.00,0.50,0.0}
\definecolor{violet}{rgb}{.40,0.00,0.7}	

\newcommand{\revision}[1]{{{#1}}}

\graphicspath{{./Figures/}}

\begin{document}

\title{Complex oscillations in the delayed FitzHugh-Nagumo equation}

\author{Maciej Krupa\footnotemark[2] \footnotemark[4]  \and Jonathan D. Touboul\footnotemark[1] \footnotemark[4]}
\renewcommand{\thefootnote}{\fnsymbol{footnote}}
\footnotetext[1]{The Mathematical Neurosciences Team, Center for Interdisciplinary Research in Biology (CNRS UMR 7241, INSERM U1050, UPMC ED 158, MEMOLIFE PSL*), jonathan.touboul@college-de-france.fr}
\footnotetext[4]{MYCENAE Laboratory, Inria Paris-Rocquencourt}
\footnotetext[2]{NEUROMATHCOMP Laboratory, Inria Sophia Antipolis-M\'editeran\'ee}

\renewcommand{\thefootnote}{\arabic{footnote}}
\date{\today}

\maketitle

\noindent \textbf{Abstract}	
Motivated by the dynamics of neuronal responses, we analyze the dynamics of the FitzHugh-Nagumo slow-fast system with delayed self-coupling. This system provides a canonical example of a canard explosion for sufficiently small delays. Beyond this regime, delays significantly enrich the dynamics, leading to mixed-mode oscillations, bursting and chaos. These behaviors emerge from a delay-induced subcritical Bogdanov-Takens instability arising at the fold points of the S-shaped critical manifold. Underlying the transition from canard-induced to delay-induced dynamics is an abrupt switch in the nature of the Hopf bifurcation. 

\medskip


\setcounter{secnumdepth}{2}
\setcounter{tocdepth}{2}

\bigskip
\bigskip
\hrule
\tableofcontents
\hrule

\bigskip

Dynamical systems with multiple timescales and delays are used as models in  a number of applications including the dynamics of excitable cells in biology~\cite{Ermentrout09, drover2004analysis, RW07, simpson2011mixed,sherman1988emergence,rinzel-ermentrout:98,terman1991chaotic}, mechanical systems~\cite{campbell-etal:09}, chemical reactions~\cite{koper1995bifurcations} and physical systems~\cite{higuera2005dynamics}. Ordinary differential equations with multiple timescales are known to display a variety of complex oscillations, including canard explosions~\cite{benoit-etal:81}, {relaxation oscillation}, mixed-mode oscillations (MMOs) \cite{brons-krupa-wechselberger:06} and bursting \cite{rinzel}. A canard explosion is a very fast transition, upon variation of a parameter, from a small amplitude limit cycle to a relaxation oscillator, a type of periodic solution consisting of long periods of quasi static {behavior} interspersed with short periods of rapid transitions. The topic of this work and the companion investigation~\cite{krupa-touboul:14a} is to understand the effect of introducing delays in a system whose dynamics, in the absence of delays, is characterized by a canard explosion and relaxation oscillations. In the present paper we focus on the FitzHugh Nagumo (FhN) system with delays:
\begin{equation}\label{eq:FhNdelay}
	\begin{cases}
			x'_t=x_t-\frac{x_t^3}{3}+y_t+ J (x_t-x_{t-\tau})\\
			y_t'=\varepsilon (a-\,x_t),
	\end{cases}
\end{equation}
for which we find a variety of complex  oscillations as delays are varied. 
System \eqref{eq:FhNdelay} is slow/fast, with the ratio of timescales given by a small parameter $\varepsilon$ and a delay $\tau\geq 0$ of arbitrary finite size in the fast equation.  It can be seen as a simple model of a neuronal network in which the delay models the typical transmission time of information through the axons and synapses (see section~\ref{sec:motivation}). Most of the dynamical mechanisms we find are relevant in a more general setting than our model system{.} 

The originality of our approach is to use the geometric singular perturbation theory (GSPT) \cite{Jones,rinzel} in the context of delayed equations. The fundamental principle of that theory, as explained in \cite{rinzel}, consists of identifying the slow and the fast approximations of the dynamics in the limit $\varepsilon=0$, and then constructing the dynamics of the full system for $\varepsilon>0$ out of these two approximations. In that theory, the first step towards the understanding the dynamics of the full system consists in understanding the fast and slow subsystems. In the case we consider the slow dynamics is simple but the fast dynamics is quite involved. A major component of our work is an exploration of the dynamics of the fast system. Based on these results, we describe the full dynamics for $\varepsilon>0$  based on the $\varepsilon=0$ approximations. To this purpose, one needs to be able to characterize the behavior of slow manifolds, which classically relies heavily on the use of Fenichel theory \cite{fenichel:79}. Recent results of Hupkes and Sandstede~\cite{hupkes-standstede:10} imply the existence of Fenichel slow manifolds in the system we study, and are  fundamental to the generalization of geometric singular perturbation theory developed in this paper as well as in the companion investigation~\cite{krupa-touboul:14a}.

Qualitatively, we will see that as delays are increased, system~\eqref{eq:FhNdelay} can undergo two types of transition from steady-state to oscillatory dynamics: one through a delay-induced canard phenomenon \cite{benoit-etal:81,baer1986singular,dumortier1996canard}, and one due a destabilization of the slow manifold due to the delays in the fast dynamics. The canard explosions found are generic: the qualitative behavior of the system in the vicinity of this transition is similar to that found in systems with no delays. Their characterization in a general context is the subject of our companion study ~\cite{krupa-touboul:14a}. Here, we concentrate on one specific system~\eqref{eq:FhNdelay} and investigate the region of the parameter space  complementary to the canard explosion, where delay induced instabilities play a significant role {in} the dynamics. While some results about Hopf bifurcations and reduction to normal form are proved, our study encompasses a part of numerical simulations and continuation. This leads us to predict and confirm by means of simulations the existence of a variety of complex oscillatory behaviors, both periodic and chaotic, including bursting orbits corresponding to recurrent dynamics composed of phases of fast oscillations interspersed with periods of quiescence. Other types of dynamics relevant to our findings are MMOs, which are recurrent trajectories of a dynamical system characterized by an alternation between oscillations of very distinct large and small amplitude (see~\cite{desroches:12} for a recent review on the subject). In our work we show the presence of a new type of MMO, due to the presence of a Bogdanov-Takens point in the fast dynamics. 

It is interesting to note that the type of complex oscillations (relaxation, MMO and bursting) found in our simple system may be relevant to the underlying biological problem, as similar solutions are reported in neuronal recordings and models~\cite{kandel2000principles,izhikevich:00}.

The interest of this work is also mathematical, as our study should be seen as a pilot investigation in the effort to extend geometric singular perturbation theory to systems with delays in the fast dynamics and finite dimensional slow dynamics.

The paper is organized as follows. We introduce in section~\ref{sec:model} our model and the differences between our approach and {previous work}.
We start by investigating, using bifurcation theory and numerical simulations, the oscillatory dynamics of the full system in section~\ref{sec:FhNFull}. We will observe a number of complex oscillatory patterns, which we investigate using the GSPT. We therefore start by investigating the bifurcations in the fast system in section~\ref{sec:fastHopf}, and complement this analytical study with extensive numerical simulations leading to a more comprehensive bifurcation diagram presented in section~\ref{sec-fastmore}. 
{S}ection~\ref{sec:Complex} presents how these elements combine {to produce} complex oscillatory patterns including MMOs, bursting and chaotic behaviors.

\section{Model and summary of the main results}\label{sec:model}
The study of the dynamics of delay differential equations with multiple timescales has been the subject of a number of important works. We briefly review these contributions and summarize the main results of the present paper in section~\ref{sec:1}. In this study we are interested in one specific system~\eqref{eq:FhNdelay} motivated by the modeling of the electrical activity of nerve cells, which we  present in more detail in section~\ref{sec:motivation}.  

\subsection{Previous works and main results} \label{sec:1}
\revision{Our results are concerned with the dynamics of a planar delay dynamical system~\eqref{eq:FhNdelay} with timescale separation and a delay in the fast equation. A similar system has been investigated in a mechanical context~\cite{campbell-etal:09}, where the behavior of a machine is described by a delay differential equation with a slow variable. The authors focus on the behavior of the system in the limit of small delays. In that regime they are able to reduce their equation {to} an ODE and use finite-dimensional methods to show existence of canard explosions. In the present paper, we shall be interested in the behavior of the system beyond the small delay regime, where the reduction to ODEs no longer holds. This is where we will need to use the singular perturbation methods to characterize the behavior of the system. 

Another class of models including fast dynamics and delays was introduced in the 1980s. First works in the context of delay differential equations with multiple timescales were {carried out}  by Mallet-Paret and Nussbaum in~\cite{mallet1986global} {and focused on} scalar equations with delays of type: 
\begin{equation}\label{eq-wesys}
\varepsilon \dot x= -x+f(x(t-1)).
\end{equation}
In that paper, they showed existence of periodic solutions and investigated their behavior in the limit $\eps\to 0$. The analyses of systems of type~\eqref{eq-wesys} were pursued by Weicker, Erneux and collaborators \cite{WE, WE1,Kozyref, WE3}. 
In these systems, setting $\varepsilon=0$ leads to a discrete dynamical system, which constitutes a significant difference with our model, where one obtains a continuous time delayed differential equation in one dimension. In that sense, our system~\eqref{eq:FhNdelay} is more directly related to finite dimensional slow-fast systems. We rely heavily on this analogy and develop the GSPT in that context. A number of results were obtained in the context of equation~\eqref{eq-wesys}: the authors of~\cite{WE, WE1}  investigated singular Hopf bifurcation, which is classically closely related to canard explosions~\cite{baer1986singular}. These authors used a method based on asymptotic expansions, which is relatively different and somehow complementary to GSPT. Finally, in the present paper we do not consider canard explosions (singular Hopf bifurcations) and focus on the parameter region where they do not occur. Our companion paper \cite{krupa-touboul:14a} does focus on canard explosions and is similar to \cite{WE, WE1} in this aspect, but rather than asymptotic expansions uses  the complementary GSPT method. {Finally we} mention that asymptotic expansions were extended to a system that has more similarities to our system in~\cite{WE3} in that the equation is no more scalar and the scaling of the delay is as in our equation~\eqref{eq:FhNdelay}. The analysis of \cite{WE3}  focuses mainly on the vicinity of Hopf bifurcation and does not make a connection between the fast bifurcation diagram and the patterns of the flow, which is fundamental to our work.

In the present paper, we shall demonstrate a few analytical results that we summarize below. These are concerned with transition to oscillation through Hopf bifurcations. The proofs are classical, yet the conclusions are very interesting, namely the lack of smoothness of of the Hopf bifurcation curve in the limit $\eps=0$ and the existence of a Bautin (generalized Hopf) point, where the sign of the first Lyapunov coefficient changes. We prove a result on a Hopf 
and Bogdanov-Takens bifurcation
in the fast system and a result on a Hopf bifurcation in the full system. The statement of the theorem relevant to  the fast system is as follows: 
\begin{theorem}\label{thm:Summary2}
	Consider the fast system of ~\eqref{eq:FhNdelay} with $J\tau>1$. 
	\begin{enumerate}
		\item There exists a family of disconnected Hopf bifurcation curves indexed by $k\in\mathbb{N}$ given by:
\begin{equation}\label{eq-forHopff}
\tau_f^k(J,a):= \frac{1}{\sqrt{a^2-1}\sqrt{2J+1-a^2}}\left(\cos^{-1}\left (\frac{1-a^2+J}{J}\right ) + 2k\pi\right) ,\qquad y=a-a^3/3.
\end{equation}

with $a\in (1,\sqrt{2J+1})$ and $\tau\in(1/J, \infty)$.
		\item There are two Bogdanov-Takens points at $J\tau=1$, $y=\pm\frac23$.
	\end{enumerate}
\end{theorem}
On the full system, we prove the following:
\begin{theorem}\label{thm:Summary1}
	Consider system~\eqref{eq:FhNdelay} with $J>0$. 
	\begin{enumerate}
		\item For $\eps>0$ there exists a unique Hopf bifurcation curve in the plane $(a,\tau)$. The curve is smooth, goes from $a=1$ to $a=\sqrt{1+2J}$ back and forth infinitely many times and reaches arbitrarily large values of $\tau$. It can be decomposed into the following branches: 
		\begin{equation}\label{eq:taustar}
			\begin{cases}
				\tau_1^{k}(J,\eps,a) :={\frac 2 {A + \sqrt{A^2+4\eps}}} \left(\cos^{-1}\left(1+\frac{1-a^2}{J}\right) + 2 k \pi \right)\\		
				\tau_2^{k}(J,\eps,a) :={\frac 2 {A - \sqrt{A^2+4\eps}}} \left(\cos^{-1}\left(1+\frac{1-a^2}{J}\right) - 2 (k+1) \pi \right)	
			\end{cases}
		\end{equation}
		with $A=\sqrt{(a^2-1)(1+2J-a^2)}$, $a\in [1,\sqrt{1+2J}]$ and $k\in \mathbb{N}$. Qualitatively, we have:
		\item If $\tau$ and $J$ are fixed so that $J\tau < 1$ then, for $\eps>0$ sufficiently small, there exists a unique Hopf bifurcation $\tau_1^0$, it is generic and supercritical.
		\item If $\tau$ and $J$ are fixed so that $J\tau > 1$ then, for $\eps>0$ sufficiently small, the Hopf bifurcation corresponding to the change of stability of the fixed point (at $\tau_1^0$) is generic and subcritical.
		\item As $\eps\to 0$, the curve $\tau_1^0$ associated to the change of stability of the fixed point converges to the union of the line segment $a=1$, $J\tau\le 1$ and the Hopf bifurcation curve of the fast system (see Theorem \ref{thm:Summary2}). The primary Hopf bifurcation $\tau_1^0$ in that limit is continuous but not differentiable, and the sub- or super-critical nature of the bifurcation changes at $J\tau=1$. 
	\end{enumerate}
\end{theorem}
Note that Theorem \ref{thm:Summary1} covers in part the parameter region $J\tau<1$, so there is some overlap with \cite{krupa-touboul:14a}, where existence of Hopf bifurcation is proved using a different approach (center manifold theorem). In \cite{krupa-touboul:14a} the existence of Hopf bifurcation is just a minor point in establishing the entire picture of canard explosion. Here we have a very simple and direct method of analysis which allows us to explain the remarkable shape of the locus of Hopf bifurcation in system~\eqref{eq:FhNdelay}. We note that as a consequence of \ref{thm:Summary1}, there exists along the Hopf bifurcation for $\eps>0$ at least one point where the first Lyapunov coefficient vanishes, i.e. a Bautin point, known also as generalized Hopf point.

Another important component of this work is the compilation of the bifurcation diagram for the fast system combining the findings of Theorem \ref{thm:Summary2} and the numerical exploration. We find that depending on the parameter values there can be one, two or three limit cycles, with a variety of bifurcations separating the corresponding regions of robust dynamics. Based on this bifurcation diagram we can conjecture a variety of complex oscillations. A point of particular interest is a dynamic BT bifurcation, studied earlier by \cite{chiba:11} in a 3D model. Near the dynamic BT point we find small amplitude chaos and MMOs. We conjecture that the existence of such solutions can be rigorously established at least in the 3D model of \cite{chiba:11}. Away from the dynamic BT points there is a variety of bursting patterns. Using the bifurcation diagram of Figure \ref{fig:FastBifs} we predict three distinct patterns: regular bursting, chaotic bursting and torus canards, whose existence we confirm numerically. }

\subsection{Motivation and model}\label{sec:motivation}
A central model in mathematical neuroscience, encompassing the excitable nature of the cells and simple enough to allow analytical developments, is given by the FitzHugh-Nagumo model~\cite{FitzHugh:55,nagumo1962active}. This equation, initially introduced as a simplification of the very versatile but much more complex Hodgkin-Huxley neuron model, describes the excitable dynamics of the membrane potential of the neuron $x$ coupled to a slow recovery variable $y$, through the slow/fast ordinary differential equations:
\begin{equation}\label{eq:FN}
	\begin{cases}
			x'=x-\frac{x^3}{3}+y\\
			y'=\varepsilon (a+b\,x-\gamma y)
	\end{cases}	
\end{equation}
where $a$ represents the input current the {neuron} is subjected to and $b$ is {the} interaction strength between the voltage and the recovery variable. The small parameter $\varepsilon$ represents the timescale ratio between the voltage and the recovery variables. Note that the classical van der Pol oscillator corresponds simply to the case $\gamma=0$~\cite{vanderpol:26}.

In~\cite{plant:81}, a linear feedback was introduced to model recurrent self-coupling. Here, we shall rather motivate our model by the description of the electrical activity of large networks. Indeed, cortical behaviors generally arise at macroscopic scales involving a large number of neurons. We consider here that neurons are electrically coupled through gap junctions, resulting in an input current proportional to the voltage difference between the communicating cells~\cite{ermentrout-terman:10b}. We also incorporate the important fact that neurons interact after a delay related to the transmission of information through the axon and dendrites. The corresponding $N$-neurons network mode with stochastic input is given by the equations:
\begin{equation}\label{eq:FNet}
	\begin{cases}
			dx_t^i=(x_t^i-\frac{(x_t^i)^3}{3}+y^i_t+\frac J N \sum_{j=1}^N (x^i_t-x^j_{t-\tau}))\,dt + \sigma dW_t^i\\
			dy^i_t=\varepsilon (a+b\,x^i_t-\gamma  y^i_t)\,dt
	\end{cases}	
\end{equation}
where $W^i_t$ are independent Brownian motions. The presence of noise allows to prove (see~\cite{touboul:12,quininao-mischler}) that in the limit $N\to \infty$, the behavior of a given neuron in the network satisfies the implicit stochastic equation:
\begin{equation}\label{eq:FNMFE}
	\begin{cases}
			dx_t=(x_t-\frac{x_t^3}{3}+y_t+ J (x_t-\Exp{x_{t-\tau}}))\,dt + \sigma dW_t\\
			dy_t=\varepsilon (a+b\,x_t- \gamma y_t)
	\end{cases}	
\end{equation}
where $\Exp{x_t}$ denotes the statistical expectation of the process $x_t$. Note that the presence of noise is essential to derive this limit. In specific cases~\cite{touboul-hermann-faugeras:11,touboul-krupa-desroches:13}, such mean-field equations can be exactly reduced to ordinary differential equations in aggregate variables (e.g., mean and standard deviation). Here, the nonlinear dynamics of the cells prevents from such a reduction. However, in the zero noise limit (a `viscosity solution', that may be seen as a fully synchronized solution of the network equation in the sense of~\cite{stewart2003symmetry}), one obtains the following self-coupled delayed FitzHugh-Nagumo equation:
\begin{equation}\label{eq:Delayed}
	\begin{cases}
			x'_t=x_t-\frac{x_t^3}{3}+y_t+ J (x_t-x_{t-\tau})\\
			y_t'=\varepsilon (a+b\,x_t - \gamma y_t).
	\end{cases}	
\end{equation}

This is precisely the equation we choose to analyze in the present manuscript. When $\tau=0$, this equation is simply the original FhN model since the interaction term vanishes. {In the limit $\eps=0$ the parameters $(a,b,\gamma)$ only affect the dynamics of the slow (reduced) equation, while the fast (layer) equation is independent of these parameters}. Among these models, we focus our attention on the simple case $\gamma=0$, and consider without loss of generality $b=-1$, which is equation~\eqref{eq:FhNdelay}. This system is particularly appealing because its equilibria have a very simple expression as a function of the parameters. This simplicity will allow to perform analytical characterization of its dynamics in the main text, and the FhN model can be understood qualitatively from the analysis of the fast equation. Numerical analysis of the system for $\gamma\neq 0$ is performed in appendix~\ref{append:FhN}. We eventually note that the system~\eqref{eq:FhNdelay}, as a delayed differential equation, is classically seen as a dynamical system in a functional space, and solutions are defined once an initial {segment} of trajectory is defined on a time interval of length $\tau$. Here, we will be interested in fixed points, periodic orbits and chaotic attractors: we will not concentrate on transient behaviors and how the trajectories of the system depends on the choice of the initial conditions.

\section{Oscillatory dynamics of the delayed FitzHugh-Nagumo equation}\label{sec:FhNFull}
Before undertaking the GSPT analysis of system~\eqref{eq:FhNdelay}, we start by investigating the fixed points of the system and their stability for $\eps>0$. In the regions were the fixed point is unstable, we investigate numerically the type of dynamics observed. 

\subsection{Hopf bifurcations}\label{sec-Hopfexist}
For any fixed $a\in \R$, there exists a unique fixed point to the FhN system~\eqref{eq:Delayed} given by $(x_a,y_a)=(a,-a+a^3/3)$. The following proposition characterizes its stability. 

\begin{prop}\label{pro:HopfFull}
	The fixed point $(a,-a+a^3/3)$ is
	\begin{itemize}
		\item stable if $\vert a \vert >\sqrt{1+2J}$
		\item unstable if $\vert a \vert<1$
		\item for $a\in [1,\sqrt{1+2J}]$, there exists one curve of Hopf bifurcations that governs changes on the stability of the fixed point. This curve is given by the branches $\tau^k_1$ and $\tau^k_2$ that together form a smooth curve along which the linearized system has at least one pair of purely imaginary eigenvalues.
	\end{itemize}	
\end{prop}
We will refine this result in section~\ref{sec:Lyapu} by showing that along the branch $\tau_1^0$ the system undergoes a Hopf bifurcation, whose sub- or super-critical type depends on the value of the parameters $(J,\tau)$ and changes along a curve approaching $J\tau=1$ when $\eps\to 0$.

\begin{proof}
	The stability of the fixed point $(a,a-a^3/3)$ is characterized by the sign of the characteristic roots $\xi$ of the system. These are found as solutions of the characteristic equation obtained by linearizing the system at the fixed point. The linearized equations are:
	\[\begin{cases}
		\dot{\tilde{x}}_t=(1-a^2 +J)\,\tilde{x}_t + \tilde{y}_t - J\tilde{x}_{t-\tau}\\
		\dot{\tilde{y}}_t = -\varepsilon \tilde{x}_t
	\end{cases}\]
	and the characteristic equation is obtained when looking for solutions of the form $(\tilde{x},\tilde{y})e^{\xi t}$. Substituting this into the linearized equation and denoting by we readily obtain:
	\[\begin{cases}
		\xi \tilde{x} =(1-a^2 +J (1-e^{-\xi\tau}))\,\tilde{x} + \tilde{y}\\
		\xi \tilde{y} = -\varepsilon \tilde{x}
	\end{cases}\]
	which has non-trivial solutions when $\det(\xi I_2 - \M)=0$ with 
	\[\M=\left(\begin{array}{cc}
		1-a^2+J(1-e^{-\xi\tau}) & 1\\
		-\varepsilon & 0
	\end{array}\right)\]
	i.e.:
	\[\xi(\xi-1+a^2-J(1-e^{-\xi\tau})) + \varepsilon =0.\]
	This equation corresponds exactly to equation~\eqref{eq:CharactRootsFast} when $\eps\to 0$. {Similarly as in the fast case this equation can be solved by elementary algebra}. In contrast with the fast system, $\xi=0$ is never a solution when $\eps>0$. This allows to write down the equations:
	\[J e^{-\xi\tau}=-\frac{\varepsilon}{\xi}-\xi + (1-a^2+J).\]
	
	Denoting $\xi=\alpha+\mathbf{i}\beta$ and taking the real part of both sides of the equality, we obtain the equation:
	\[Je^{-\alpha\tau}\cos(\beta\tau) = -\frac{\eps \alpha}{\alpha^2+\beta^2}-\alpha +(1-a^2+J).\]
	If $\alpha>0$, we therefore necessarily have 
	\[-1\leq e^{-\alpha\tau}\cos(\beta\tau)\leq 1+\frac{1-a^2}{J}\]
	and there is no possible pairs of real values $\alpha>0,\beta \in \R$ satisfying this inequality for $a^2 > 1+2J$. Similarly, if $\alpha<0$, taking the modulus on both parts of the equality, we can write down the inequality:
	\[J>Je^{-\alpha\tau} = \sqrt{\left(1-a^2+J - \alpha (\frac{\eps}{\alpha^2+\beta^2}+1)\right)^2 +\beta^2\left(\frac{\eps}{\alpha^2+\beta^2}-1\right)^2 }\geq (1-a^2+J)\]
	which is not possible for $a^2<1$.
	
	The stability of the fixed point necessarily changes within the interval $[1,\sqrt{1+2J}]$, and therefore characteristic roots, which are continuous with respect to the parameters of the system, cross the imaginary axis. Since we already noted that $\xi=0$ is not a possible root, changes of stability necessarily occur when the system  has purely imaginary solutions $\xi=\mathbf{i}\zeta$, that hence satisfy the complex equation:
	\[J e^{-\mathbf{i}\zeta\tau}=\mathbf{i}(\frac{\varepsilon}{\zeta}-\zeta) + (1-a^2+J).\]
	Taking the modulus, one finds, as expected, that solutions exist for any $\vert a\vert \in [1,\sqrt{1+2J}]$ (and only in that interval), and in that case we find two possible characteristic roots (together with their complex conjugate). We choose by convention the two roots corresponding to $\eps/\zeta - \zeta = A$:
	\[\begin{cases}
		\zeta_1=\frac 1 2 (-A-\sqrt{A^2+4\eps})\\
		\zeta_2=\frac 1 2 (-A+\sqrt{A^2+4\eps}).
	\end{cases}\]
	With these expressions of the eigenvalues, we can now solve the dispersion relationship and find the parameters for which Hopf bifurcations occur. Taking the real part, we find that the delay can take the two following values:
	\[\tau_{\pm} = \pm \frac{1}{\zeta} \cos^{-1}\left(1+\frac{1-a^2}{J}\right) + \frac{2k\pi}{\zeta}\]
	with $\zeta=\zeta_1$ or $\zeta=\zeta_2$ and $k\in \mathbb{Z}$. The imaginary part of the dispersion relationship provides an additional equation that is only satisfied when $\tau=\tau_-$. 
	This leads to the formulae~\eqref{eq:taustar} where the $k$ have been fixed to ensure positivity of the curves. It remains to show that these two families of curves actually form together a unique smooth curve in the parameter plane $(a,\tau)$ that tends to infinity as visible in Figure~\ref{fig:snaking}. We demonstrate this fact by showing that the expansion of the different branches of the curve match at second order at the turning points $a\in \{1,\sqrt{1+2J}\}$. At $a=1$, we have:
	\[\begin{cases}
		\tau_1^{k}(J,\eps,a) = \frac{2k\pi}{\sqrt{\eps}} + \frac{K_1^k}{\sqrt{\eps}}\sqrt{a-1}+\mathcal O (a-1)\\
		\tau_2^{k}(J,\eps,a) = \frac{2(k+1)\pi}{\sqrt{\eps}} - \frac{K_1^{k+1}}{\sqrt{\eps}}\sqrt{a-1}+\mathcal O (a-1)
	\end{cases}\]
	with $K_1^k=2/\sqrt{J}-2Jk\pi/\sqrt{\eps}$. This implies that the union of $\tau_1^{k+1}$ and $\tau_2^{k}$, seen as a function of $\tau$, is the graph of a map at least twice continuously differentiable at $\tau=\frac{2(k+1)\pi}{\sqrt{\eps}}$. Similarly, at $a^{*}=\sqrt{1+2J}$ we have:
	\[\begin{cases}
		\tau_1^{k}(J,\eps,a) = \frac{(2k+1)\pi}{\sqrt{\eps}} + \frac{K_2^k}{\sqrt{\eps}}\sqrt{a-a^*}+\mathcal O (a-a^*)\\
		\tau_2^{k}(J,\eps,a) = \frac{2(k+1)\pi}{\sqrt{\eps}} - \frac{K_2^k}{\sqrt{\eps}}\sqrt{a-a^*}+\mathcal O (a-a^*)
	\end{cases}\]
ensuring a smooth connection between $\tau_1^{k}$ and $\tau_2^{k}$.  We eventually notice that $\tau_2^k$ belongs to the interval $[2k+1,2k+2]\pi/\sqrt{\eps}$ and therefore disappear at infinity in the limit $\eps\to 0$. In contrast, $\tau_1^k \leq \tau_f^k$ the $k^{th}$ branch of Hopf bifurcations of the fast system given by equation~\eqref{eq-forHopff} and announced in theorem~\ref{thm:Summary1}: as $\eps\to 0$, the two families become infinitely close in any compact contained in $(1,\sqrt{1+2J})$. We conclude that $\tau_1^0$ is the primary Hopf bifurcation curve for $\tau<\pi/\sqrt{\eps}$, i.e. for fixed $a$, the first value of the delay leading to a change of stability of the fixed point. It smoothly connects $\tau=0$ at $a=1$ to $\tau=\pi/\sqrt{\eps}$ at $a=\sqrt{1+2J}$, where it has a turning point and merges smoothly with $\tau_2^0$ at $a=a^*$, a branch ending at $a=1$ and $\tau=2\pi/\sqrt{\eps}$, and the matching persists for increasing values of $k$:
\begin{itemize}
		\item $\tau_1^k$ connect $2k\pi/\sqrt{\eps}$ at $a=1$ to $(2k+1)\pi/\sqrt{\eps}$ at $a=\sqrt{1+2J}$
		\item $\tau_2^k$ that connects $(2k+1)\pi/\sqrt{\eps}$ at $a=\sqrt{1+2J}$ to $(2k+2)\pi/\sqrt{\eps}$ at $a=1$
		\item $\tau_1^{k}$ and $\tau_2^k$ smoothly connect at $a=\sqrt{1+2J}$ and $\tau_2^k$ smoothly connects to $\tau_1^{k+1}$ at $a=1$.
	\end{itemize}
The figure~\ref{fig:snaking} moreover shows that the curve forms a sequence of loops, visible for $\eps$ large enough. Sufficiently close to the boundary $a=\sqrt{1+2J}$, increasing $\tau$ leads to a series of Hopf bifurcations changing the stability of the fixed point. But as $\eps\to 0$, all this complexity escape to infinity, and the curve $\tau_1^0$ essentially controls the stability of the fixed points for finite delays. 
\end{proof}
	
\begin{figure}
	\centering
		\subfigure[$\eps=0.01$]{\includegraphics[width=.4\textwidth]{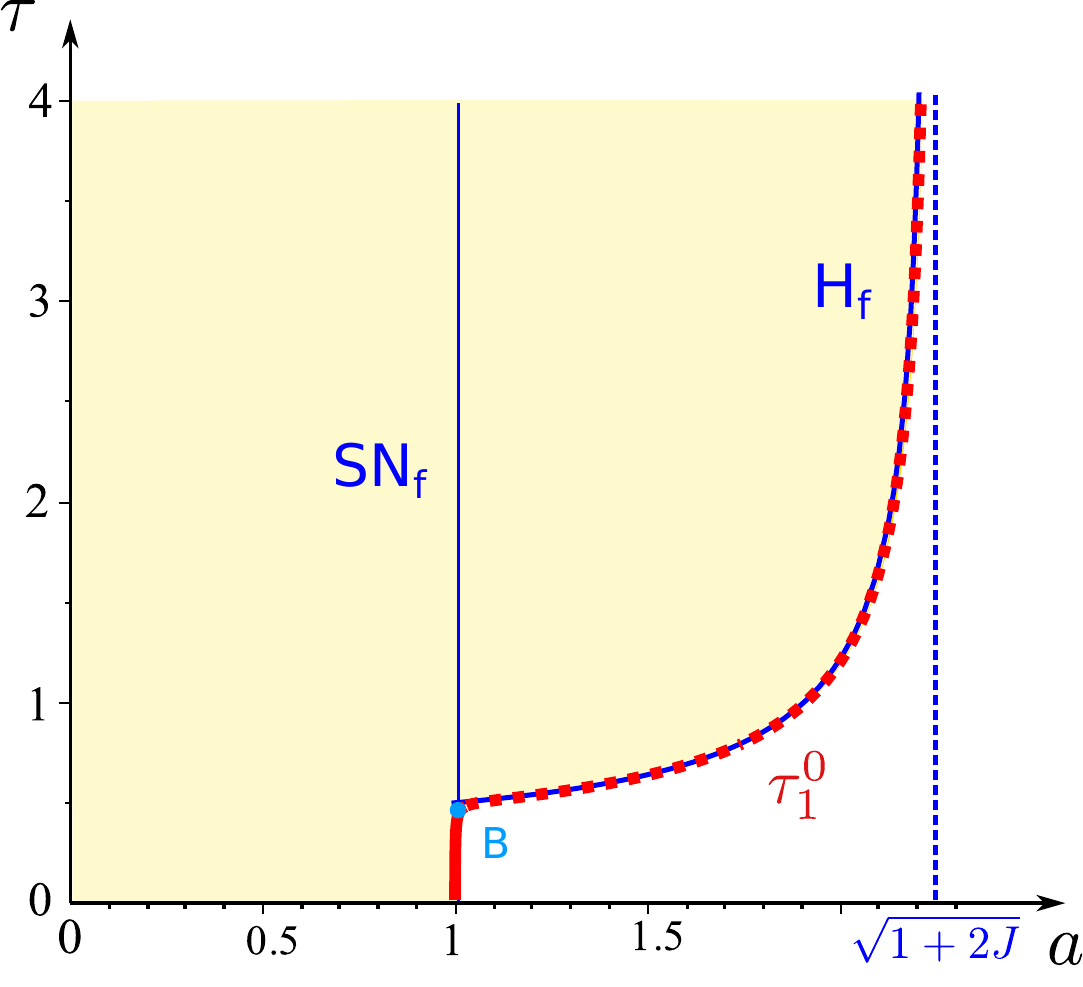}}
		\subfigure[$\eps=2$]{\includegraphics[width=.55\textwidth]{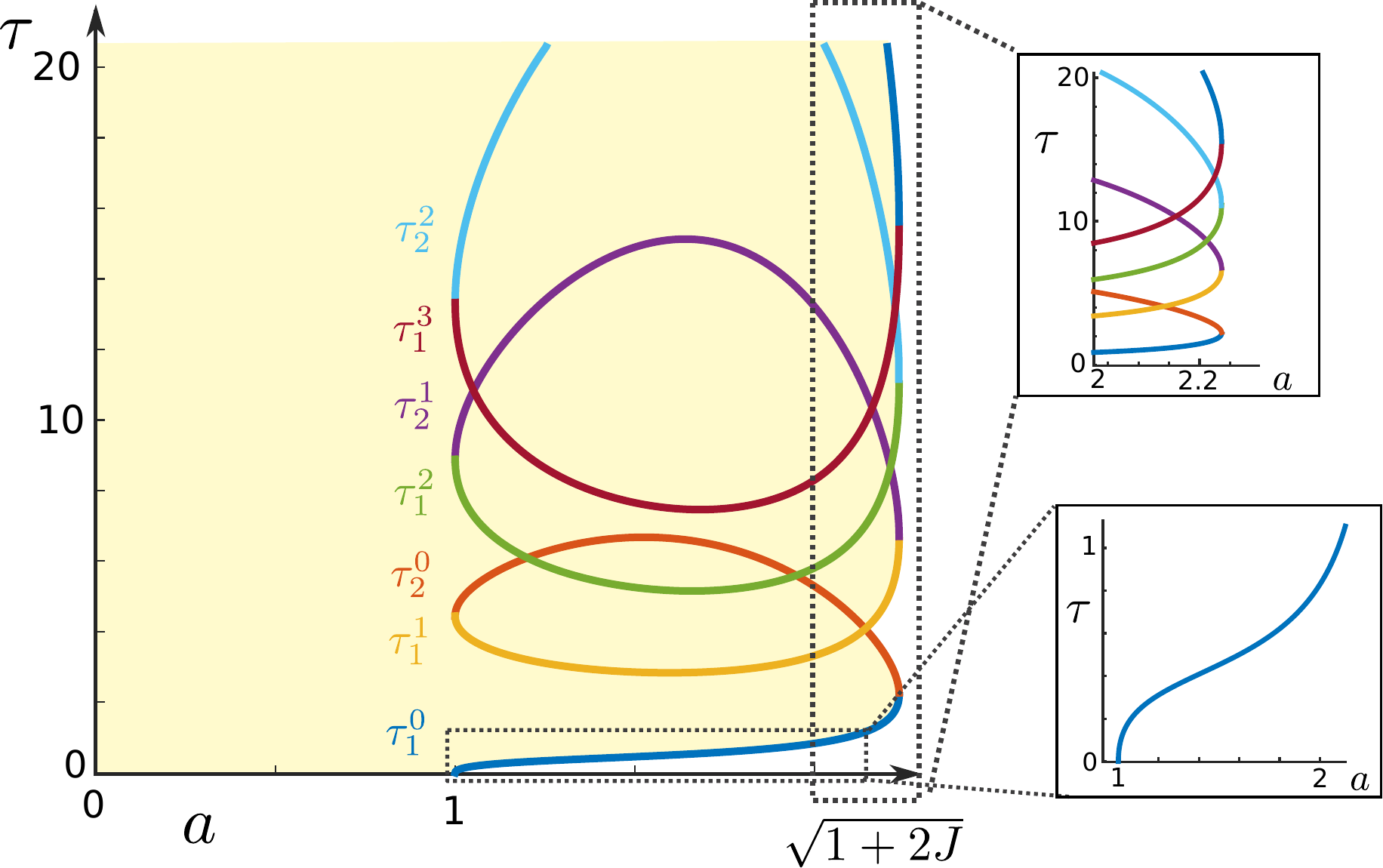}\label{fig:snaking}}
	\caption{Hopf bifurcations for system~\eqref{eq:FhNdelay} as a function of $(a,\tau)$ with $J=2$ and  (left) $\eps=0.01$: the red curve represents $\tau_1^0$ which is the only visible Hopf bifurcation in that range of parameters. The Hopf bifurcation is supercritical below the Bautin point B (solid line) and subcritical above (dotted line). Superimposed is the bifurcations of the fast system~\eqref{eq:FastDelayed}  with $J=2$ (blue curves). SNf: saddle node, Hf: subcritical Hopf, closely approaching $\tau_1^0$ on $J\tau<1$ and $J\tau>1$ respectively. (right): $\eps=2$ shows how the different branches of Hopf bifurcation form a single smooth curve. In the range of values plotted, we see $\tau_1^k$ and $\tau_2^k$ for $k\in\{0\cdots 3\}$.  }
	\label{fig:HopfTotalSystem}
\end{figure}

\subsection{Criticality of Hopf bifurcations}\label{sec:Lyapu}

\revision{

In this section we compute the first Lyapunov coefficient $\lambda$ of the normal form at the primary Hopf bifurcation at $\tau=\tau_1^0(J,\eps,a)$. We show the following transition:
\begin{prop}\label{pro:LyapuTrans}
	In the limit $\eps\to 0$, the Hopf bifurcation is supercritical for $J\tau<1$ and subcritical otherwise. 
\end{prop}
\begin{proof}
	We first focus on the case when $J\tau<1$. Note that in this case the Hopf bifurcation is found very close to the line
	$a=1$. Consider the extended system obtained from \eqref{eq:FhNdelay} by adding the equation $\eps'=0$. For this system there a center manifold reduction, based on the classical theory developed, for example, in~\cite{faria-magalhaes:95}.
	We omit the straightforward details of this computation, that lead to show that on the center manifold, the reduction followed by a reflection in the $\tilde x$ variable, yields the following system on the center manifold:
	\begin{equation}\label{eq-syscm}
	\begin{cases}
		\der{\tilde{x}}{\theta} =  -\tilde{y}+\tilde{x}^2-\frac{\tilde{x}^3}{3}+a_1\eps\tilde x+ \mathcal O (\tilde y^2, \tilde x^2\tilde y, \tilde x^4)\\
		\der{\tilde{y}}{\theta} = \varepsilon (\tilde{x}-\tilde{a}+ \mathcal O (\eps\tilde y^2, \eps\tilde x\tilde y,\eps\tilde x^2)),
	\end{cases}
	\end{equation}
	where $\tilde a=1-a$,
	\begin{equation}\label{eq-defa1}
	a_1=\frac{J\tau^2}{2(1-J\tau)},
	\end{equation}
	If the higher order terms are omitted \eqref{eq-syscm} differs from the classical van der Pol system by the term $a_1\eps\tilde x$. Note that $a_1$ is positive
	and blows up as $\tau$ approaches $1/J$. 

	System \eqref{eq-syscm} has a canard point (singular Hopf point) at an associated Hopf bifurcation, which is the same as computed in Section \ref{sec-Hopfexist}.
	A straightforward application of formula (3.4.11) in \cite{guckenheimer-holmes:83} yields the first Lyapunov coefficient for this Hopf  bifurcation:
	\begin{equation}\label{eq-LiapcoJtau<1}
	\lambda = -\sqrt\eps/8+ \mathcal O (\eps).
	\end{equation}
	Hence the bifurcation is supercritical and yields a stable limit cycle. 

	Concerning the case $J\tau >1$ we will compute in Section \ref{sec:fastHopf} the Lyapunov coefficient in the context of  Hopf bifurcation in the fast subsystem of \eqref{eq:FhNdelay}, obtaining that this coefficient is positive.
	The work of the \cite{kirk_wechselberger} (see also~\cite{guckenheimer-osinga}) shows that the fast Hopf bifurcation perturbs to a nearby Hopf bifurcation of the full system with $\eps>0$, yet the Lyapunov coefficient of this Hopf bifurcation 
	may not converge to the Lyapunov coefficient of the fast Hopf bifurcation as $\eps\to 0$.  This discontinuity relies on the presence of quadratic terms in the slow equation. Note that for \eqref{eq:FhNdelay} no such terms are present
	and that when $\eps=0$ the slow variable is a parameter, so that the center manifold reduction does not introduce quadratic terms without a factor of $\eps$ (see the remainder terms in \eqref{eq-syscm}). 
	Hence it is easy to check using the approach \cite{kirk_wechselberger} that in this case the Lyapunov coefficient is continuous in the limit $\eps\to 0$. We conclude that Hopf bifurcation for $J\tau>1$ is subcritical and leads to unstable periodic orbits.
	
\end{proof}

For $\eps>0$ it is possible to compute the Lyapunov coefficient $\lambda$ directly using the normal form reduction, but the expressions become complicated. Aided by formal calculation software (Maple\textregistered), we did the calculation of the first Lyapunov coefficient along the Hopf bifurcation line $\tau_1^0(J,\eps,a)$ following the classical theory~\cite{faria1995normal,campbell2009calculating}. Details are provided in appendix~\ref{append:NormalForms}. {The very complicated expression of the Lyapunov coefficient obtained in this manner is} evaluated numerically and confirms that there exists $\tau_s(\eps,J)>0$, approaching $1/J$ for small $\eps$, such that for $\tau<\tau_s(\eps,J)$, the bifurcation is supercritical and for $\tau>\tau_s(\eps,J)$, it is subcritical. A plot of the Lyapunov coefficient for $\eps=0.1$ as a function of $J$ and $\tau$ appears in Figure~\ref{fig:Lyapunov} (A). The change of sign of the Lyapunov coefficient occurs along a line $\tau_s$ closely following $\tau=1/J$. For $\tau<1/J$, the Lyapunov coefficient is negative and of small absolute value. It then sharply jumps towards higher values before slowly returning to smaller positive values. This is visible in the zoom of the Lyapunov coefficient for $J=2$. This is perfectly consistent with the properties that: (i) the Lyapunov coefficient is of order $\sqrt{\eps}$ for $J\tau<1$, and (ii) the remark that the Lyapunov coefficient of the fast system tends to infinity at $J\tau \to 1$.  
Eventually, we note that at $\tau=\tau_s$, the the first Lypaunov coefficient changes sign in a non-degenerate manner, and the map that associates to $\tau$ the eigenvalues of the Jacobian matrix and the first Lyapunov coefficient is regular. We conjecture that the system undergoes a Bautin bifurcation at this point. In order to rigorously demonstrate the presence of this bifurcation in the system, one needs to compute the second Lyapunov coefficient at this point and show that it is not zero: this amounts to extending the analysis performed in appendix~\ref{append:NormalForms} to access terms of order $5$ of the reduction of the system on the center manifold, leading to tedious computations. 
}

\begin{figure}[h]
	\centering
		\includegraphics[width=.7\textwidth]{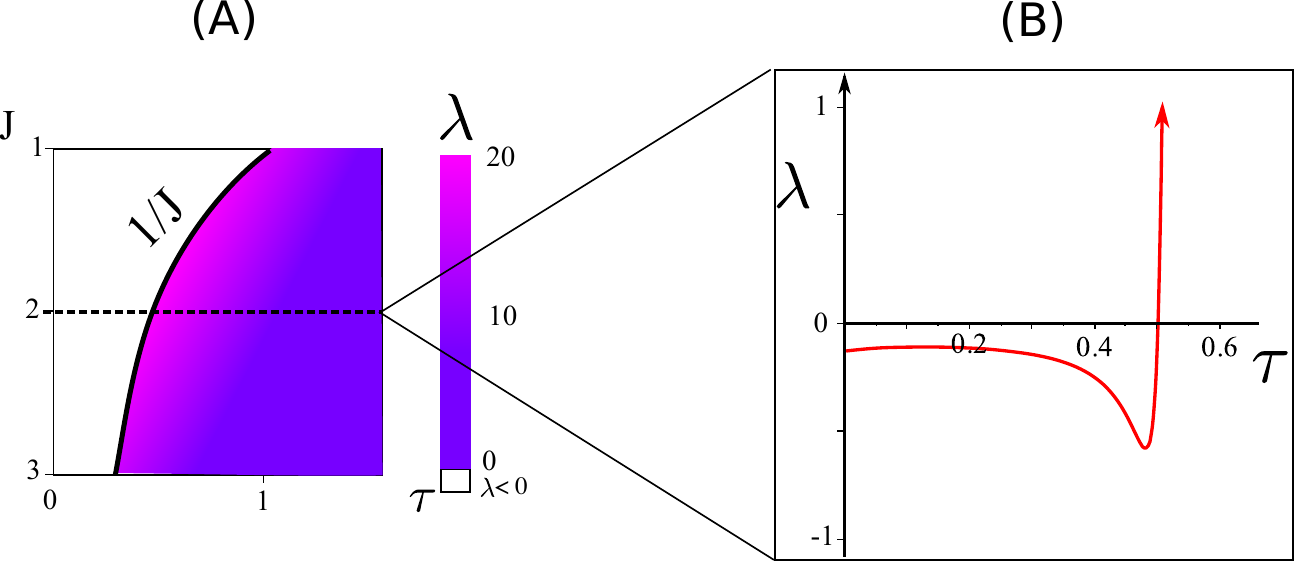}
	\caption{{First Lyapunov coefficient of the Hopf bifurcation at $\tau=\tau_1^0(J,\eps,a)$ with $\eps=0.01$. (A) Lyapunov coefficient as a function of $J$ and $\tau$. Negative values are plotted in white: the Lyapunov coefficient changes sign along a curve closely approximating $\tau=1/J$. (B) We zoomed for $J=2$ on the negative part of the curve of Lyapunov coefficients. We observe the sharp change occurring close to $0.503$ and disappears to positive values (red arrow).}}
	\label{fig:Lyapunov}
\end{figure}
\subsection{The nature of periodic orbits: numerical explorations}
The presence of Hopf bifurcations marks the transition between regions of parameters where the fixed point is stable or unstable: in particular, in the parameter region depicted in yellow in Figure~\ref{fig:HopfTotalSystem}, the fixed point is unstable, and oscillations or chaotic solutions arise. In this section we give a preview of what can happen as $\tau$ grows for fixed $a>1$, $\eps>0$ and $J>0$. The topic of section \ref{sec:Complex} will be to use the GSPT to uncover the dynamical origin of these phenomena. 
Distinct qualitative behaviors as $\tau$ is increased, summarized in Figure~\ref{fig:AllBehaviors}. 

For values of the delay smaller than the value corresponding to the Hopf bifurcation ($\tau= \tau_1^0(J,\eps,a)\simeq 0.429$), the fixed point $(v_a,w_a)$ remains stable, and precisely at crossing the Hopf bifurcation line, a small cycle appears. In a very small range of delay values, the cycle {can} become more complex, indicating the possible presence of period doubling bifurcations. A cascade of such bifurcations seem to lead to a chaotic behavior as delays are further increased (although in a very small neighborhood of the Hopf bifurcation), as depicted in Figure~\ref{fig:AllBehaviors}-\emph{Small Oscillations}.

As the delay is further increased { a sudden switch to large oscillations takes place, a} hallmark of the presence of canard {transitions} in the system. These oscillations are more complex than relaxation oscillations emerging in two-dimensional slow-fast systems. They show the typical shape of Mixed-Mode oscillations (see Figure~\ref{fig:AllBehaviors}-\emph{MMOs}), with three distinct phases: one small oscillation related to the ghost of the small periodic orbit (highlighted in blue boxes in the figure, `classical' MMO small oscillations), small oscillations related to the properties of the convergence of the delayed fast equation towards the critical manifold (red boxes), and large relaxation oscillations corresponding to switches between the two attractive parts of the critical slow manifold. We will study this phenomenon more in details in section~\ref{sec:MMOs}.
\begin{figure}[!h]
	\centering
		\includegraphics[width=.9\textwidth]{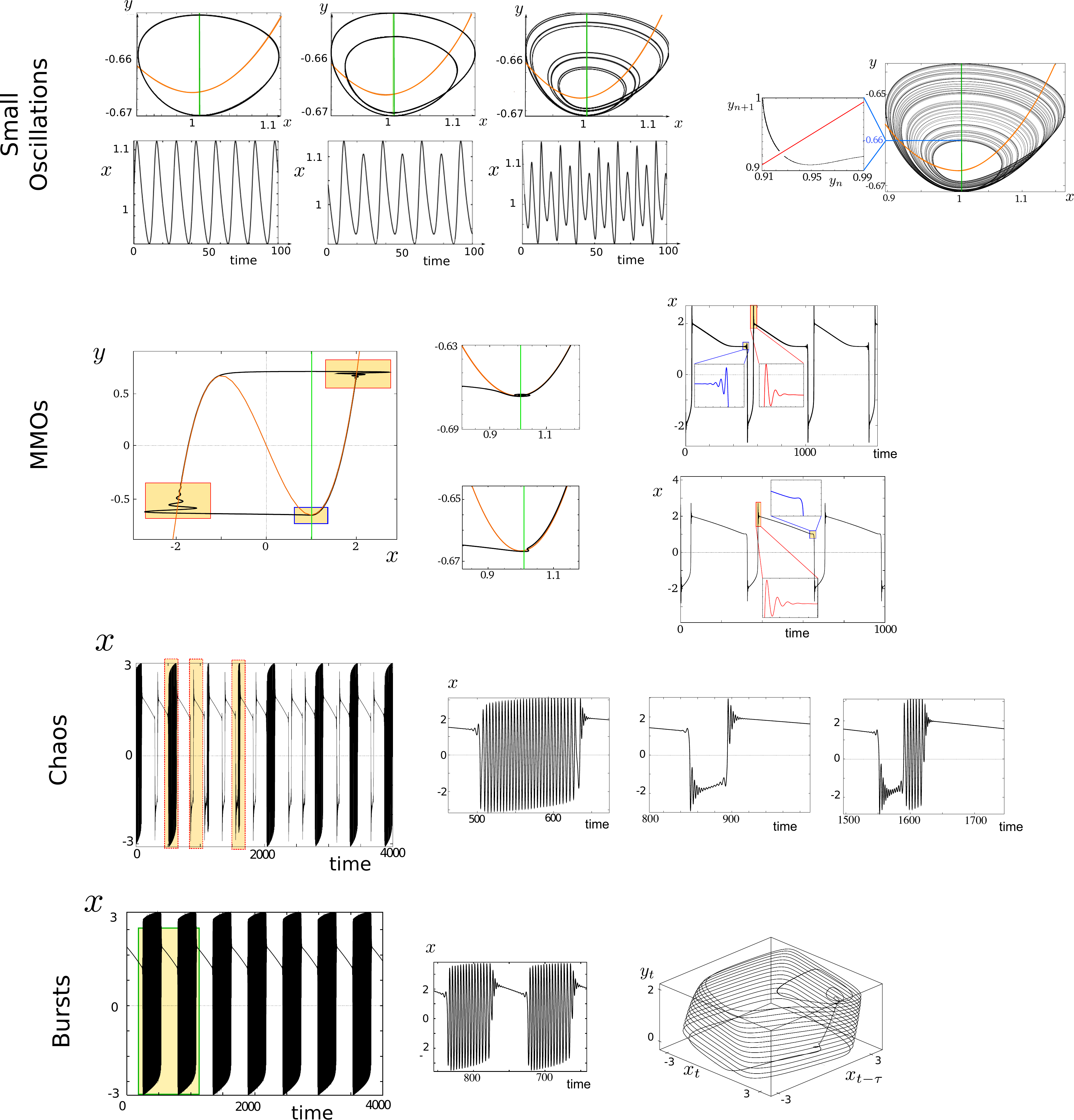}
	\caption{Trajectories of the full system~\eqref{eq:FhNdelay} with $a=1.01$, $J=2$ and $\varepsilon=0.05$ as delays are increased. (Small Oscillations) from left to right, $\tau=0.4$, $\tau=0.401$ and $\tau=0.408$: The cycle arising from the Hopf bifurcation undergoes period doubling bifurcations and chaos for $\tau=0.41$, as illustrated by the Ruelle plot on a Poincar\'e section (blue line at $y=-0.66$). (MMOs) $\tau=0.51$ (left) $\tau=0.55$ (right). (Chaos) $\tau=0.7$: time series shows irregular alternation of bursts, MMOs and mixed behaviors between MMO and burst, (Bursts) $\tau=1$. }
	\label{fig:AllBehaviors}
\end{figure}

As the delay is further increased, these MMOs suddenly disappear in favor of bursting periodic orbits (see Figure~\ref{fig:AllBehaviors}-\emph{(Bursts)}), characterized by the presence of very fast oscillations interspersed with periods of slow behavior. The analysis of these orbits will be performed in section~\ref{sec:Bursting}, of which we display time series and trajectories in the extended phase space $(x_t,y_t,x_{t-\tau})$. 

We {discuss} also the behavior found at the transition from MMOs to bursts, {where} the system displays extremely wild chaotic behaviors characterized by irregular switches between MMOs and bursts (Figure~\ref{fig:AllBehaviors}-\emph{(Chaos)}). These phenomena will be accounted for in section~\ref{sec:ChaosTransition}. In order to illustrate the chaotic nature of the trajectory, it is generally convenient to display a Ruelle plot, as we did for the chaotic small oscillation. But here, Ruelle plots on Poincar\'e section in the axes $(x,y)$ are scattered over very complex attractors that do not seem to be one-dimensional, and do not provide a concise information of the chaotic nature of the trajectory. Rather, we found that the value of $x$ at time $t-\tau$ at crossing a Poincar\'e section in $y$ shows a lower-dimensional behavior that we display in Figure~\ref{fig:Poincare}.

\begin{figure}[h]
	\centering
		\includegraphics[width=.25\textwidth]{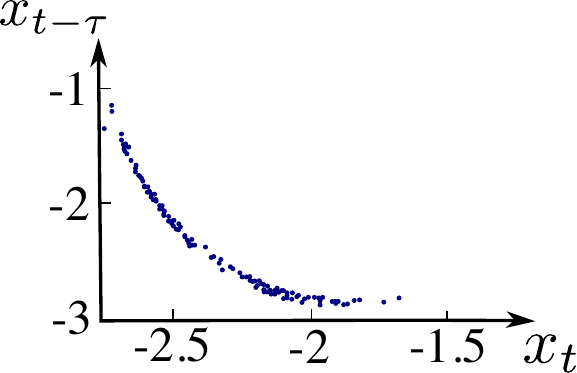}
	\caption{Sequence of values of $x_{t_n-\tau}$ where $t_n$ are the consecutive times the trajectory crosses the Poincar\'e section $y=-0.4$. Parameters $J=2$, $\eps=0.05$ and $\tau=0.7$. }
	\label{fig:Poincare}
\end{figure}

\section{Hopf bifurcations in the fast dynamics}\label{sec:fastHopf}
In order to {explain} the complex dynamics presented in the previous section, the singular perturbation theory {takes into the account} properties of both the fast and the slow flow in the limit $\varepsilon = 0$.
In the case of \eqref{eq:FhNdelay} the slow dynamics is very simple and the challenge is to understand the fast dynamics given by the solution of equation~\eqref{eq:Delayed} in the limit $\eps = 0$:
\begin{equation}\label{eq:FastDelayed}
			x'_t=x_t-\frac{x_t^3}{3}+y+ J (x_t-x_{t-\tau})
\end{equation}
The slow variable $y$ is now looked at as a parameter of the dynamics. 
An important object are families of equilibria of \eqref{eq:FastDelayed}, known also as critical manifolds, as they give rise 
to slow manifolds of \eqref{eq:FhNdelay}. In this section we identify the critical manifolds of \eqref{eq:FastDelayed}, characterize their stability and find non-hyperbolic bifurcation points (Hopf) induced by the presence of delays. 

\subsection{The critical manifold}
The first step in analyzing the fast equation is to characterize its fixed points (critical manifold). These are given by the solutions to the algebraic equation:
\begin{equation}\label{eq:FPFast}
	x-\frac{x^3}{3}+y=0.
\end{equation}
We therefore have three fixed points when $\vert y \vert < \frac 2 3$ and one fixed point otherwise. A closed from expression is obtained using Cardano's method:
\begin{itemize}
	\item for $\vert y\vert >2/3$, the unique solution is given by:
	\[x_0=\left(\frac{3y+\sqrt{9 y^2 -4}}{2}\right)^{1/3}+\left(\frac{3y-\sqrt{9 y^2 -4}}{2}\right)^{1/3}\]
	\item  for $\vert y \vert < 2/3$, $\Delta<0$, the three solutions are given by
	\[x_k=2\cos\left(\frac 1 3 \cos^{-1}\left(\frac{3y}{2}\right) + \frac{2k \pi}{3}\right) \qquad, \qquad k=0,1,2,\]
	that we order and denote $x_-(y)<x_0(y)<x_+(y)$. 
	\item Eventually, for $ y=\pm2/3$, we have a double root $x=\mp 1$ and a simple root $x=\pm 2$. 
\end{itemize} 
The critical manifold is hence made of three branches of fixed points: the branch $x_+(y)\geq 1$ defined for $y\geq -2/3$, $x_-(y)\leq -1$ for $y\leq 2/3$, and $x_0(y) \in [-1,1]$ for $y\in[-2/3,2/3]$.

\subsection{Stability of fixed points on the critical manifold}
In the absence of delays, $x_0(y)$ is always a saddle fixed point while $x_{\pm}(y)$ are stable fixed points, and the fast equation has saddle-node bifurcations at $y=\pm \frac 2 3$. The presence of the delay term, which did not affect the shape of the critical manifold, may affect its stability. The stability of the fixed points on the critical manifold with respect to the dynamics of \eqref{eq:FastDelayed} is characterized in  the following:
\begin{prop}\label{pro:bifurcFast}
	System \eqref{eq:FastDelayed} undergoes the following bifurcations:
	\begin{itemize}
		\item saddle-node bifurcations at $y=\pm \frac 2 3$ that correspond to parameters for which the fixed points $x_\pm (y)$ collide with $x_0(y)$. These bifurcations are independent of  $\tau$.
		\item For any $y \in [-\frac 2 3, \frac 2 3 (J-1)\sqrt{1+2J}]$, there exists a unique value of $\tau$ given by $\tau_f^0(x_{\pm}(y))$ (recall the definition of  $\tau_f^k$ in equation~\eqref{eq-forHopff}) above which the fixed point $x_{\pm}(y)$ loses stability through a generic subcritical Hopf bifurcation. Secondary Hopf bifurcations arise for $\tau=\tau_f^k(a)$ for any $k\in \N\setminus \{0\}$.
	\end{itemize}
	These bifurcations organize the stability of the three fixed points. For $J>0$, their stability is as follows
	\begin{itemize}
		\item For any $\vert y \vert<2/3$ fixed point $x_0(y)$ is a saddle 
		\item For $\vert y \vert > \frac 2 3 (J-1)\sqrt{1+2J}$, the fixed points $x_{\pm}(y)$ are stable regardless of the value of $\tau$
		\item For $\vert y \vert < \frac 2 3 (J-1)\sqrt{1+2J}$, the fixed points $x_{\pm}(y)$ are stable for $\tau <\tau (x_{\pm}(y))$, and unstable otherwise.
	\end{itemize}
\end{prop} 

The Hopf and saddle-node bifurcation curves are depicted for $J=2$ in Figure~\ref{fig:BifurcationsFast}. 

\begin{figure}
	\centering
		\subfigure[Bifurcations fast system in $(y,\tau)$]{\includegraphics[width=.4\textwidth]{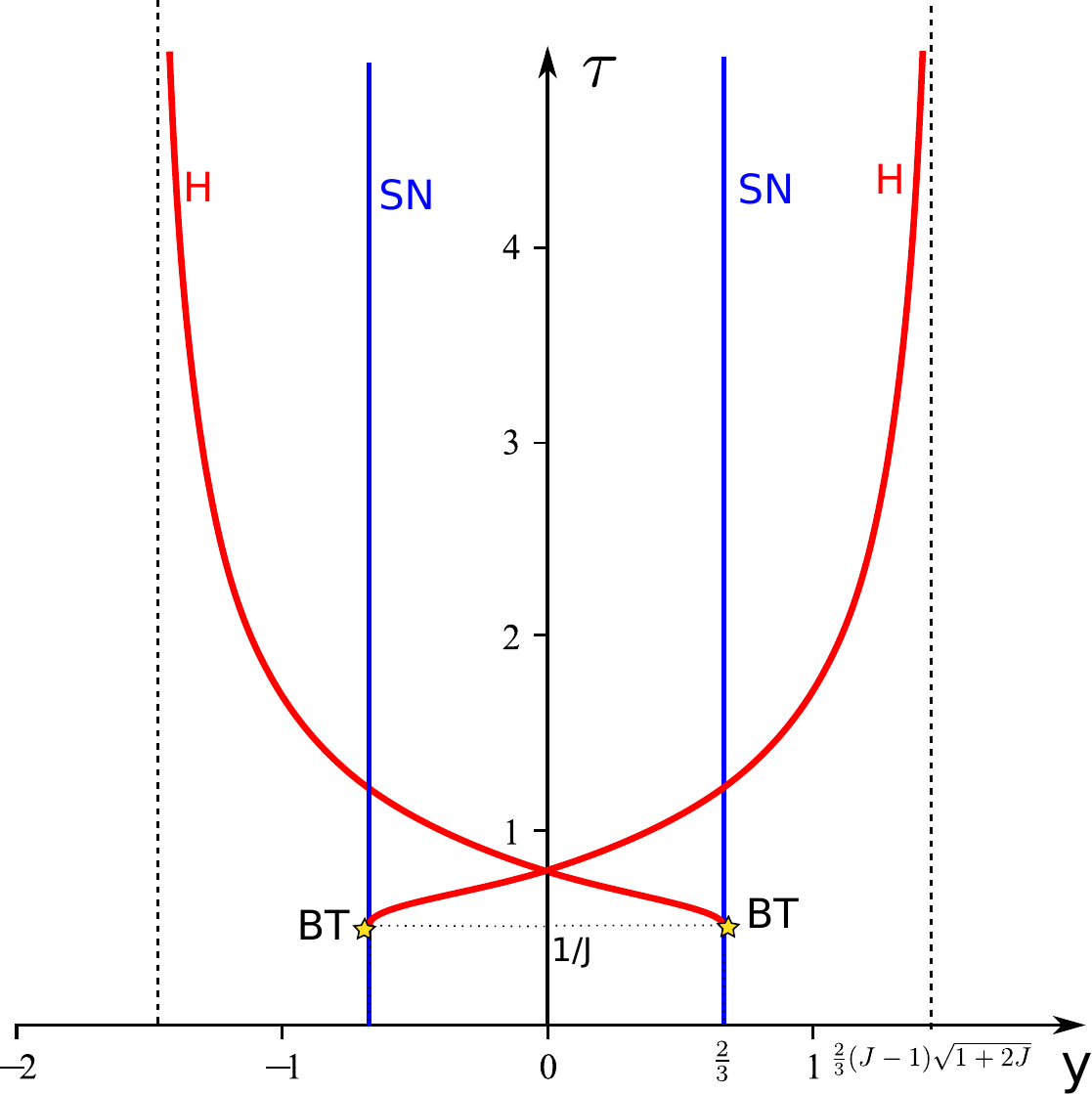}}\qquad 
		\subfigure[Bifurcations fast system in $(a,\tau)$ with $a=x_+(y)$]{\includegraphics[width=.4\textwidth]{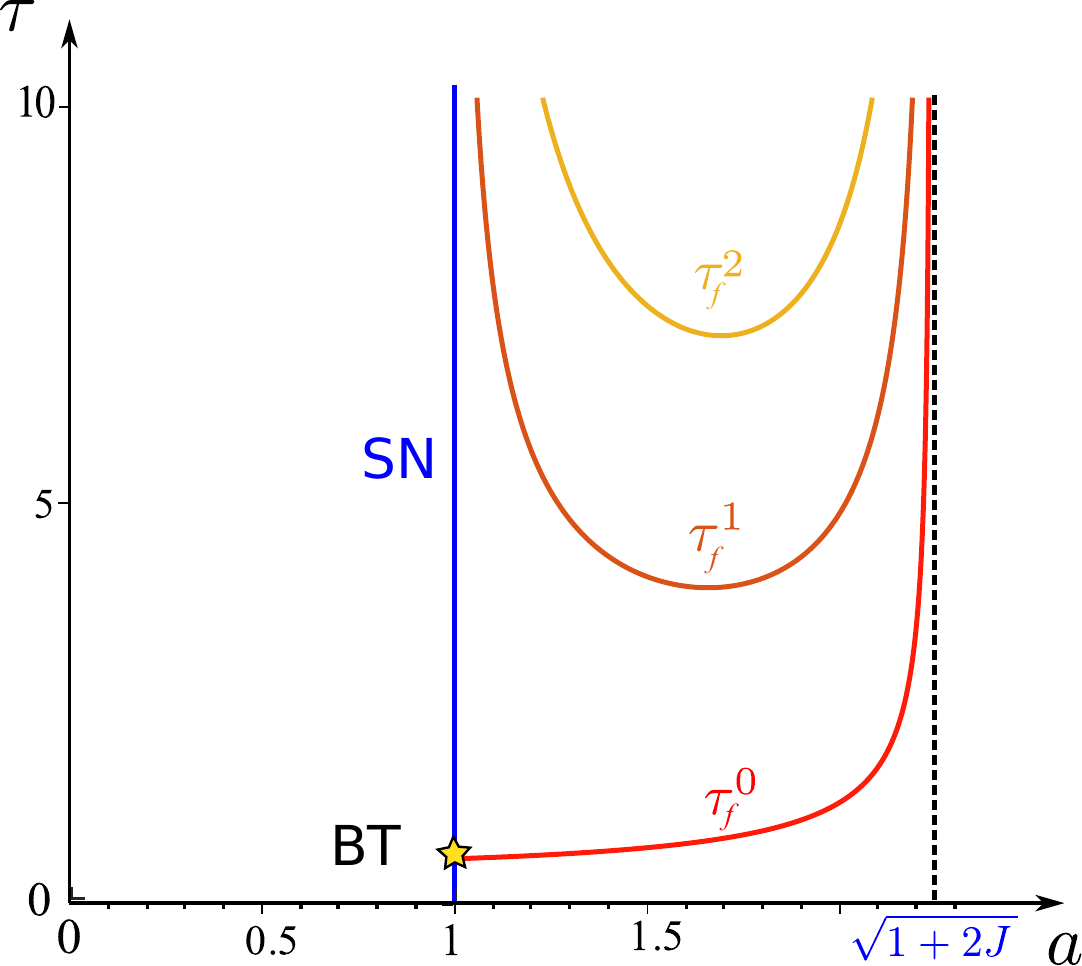}}
	\caption{Hopf bifurcation for the fast system~\eqref{eq:FastDelayed} for $J=2$. (a) The branches of solutions $x_\pm (y)$ undergo saddle-node bifurcations at $y=\pm 2/3$ (blue lines) and Hopf bifurcations along the red curves $\tau=\tau_f^0$ with vertical asymptotes at $y=\pm 2/3$ and $y= \pm \frac{2}{3}\sqrt{1+2J}(J-1)$ (black dashed line). (b) Primary and secondary Hopf bifurcations in the plane $(a,\tau)$ with $a=x_+(y)$. The visible curves in the range plotted are $\tau_f^k$ for $k\in\{0,1,2\}$. The BT star identify the generic subcritical Bogdanov-Takens bifurcations characterized in theorem~\ref{thm:BT}.}
	\label{fig:BifurcationsFast}
\end{figure}

\begin{proof}
We start by characterizing the stability of the fixed points away from the bifurcation points announced in the proposition. To this end, we investigate the sign of the real part of the characteristic roots of the system at a given fixed point $x^*$. These are the solutions $\xi$ of the equations:
	\begin{equation}\label{eq:CharactRootsFast}
		\xi=1-(x^*)^2 +J -Je^{-\xi\tau}
	\end{equation}
	This equation may be solved using special functions\footnote{Indeed, the solutions to the characteristic equation are given by the Lambert functions $W_k$ (the different branches of the inverse of $x\mapsto xe^x$, see e.g.~\cite{corless:96}):
	\[\xi = A+\frac 1 \tau W_k\left(-\tau J e^{-\tau A}\right)\]
	with $A=1-(x^*)^2+J$. The stability of $x^*$ is hence governed by the sign of the real part of the rightmost eigenvalue, given by the real branch $W_0$ of the Lambert function, and if the argument of the Lambert function has a real part greater than $-e^{-1}$ the root is unique. If not, we have two eigenvalues with the same real part corresponding to $k=0$ or $-1$.}, but elementary calculus can also be used to characterize the bifurcation curves in the parameter space.
	
Indeed, from the complex characteristic roots equation~\eqref{eq:CharactRootsFast}, we can show that $x_0(y)$ is unstable whatever the parameters $(J,\tau)$ of the system, and $x_{\pm}(y)$ are stable when $\vert y\vert>\frac 2 3 (J-1)\sqrt{1+2J}$. Indeed, if $\xi=\alpha+\mathbf{i}\beta$, we can deduce from~\eqref{eq:CharactRootsFast} that:
	\[Je^{-\alpha\tau}=\sqrt{\left(-\alpha+(1-(x^*)^2+J)\right)^2+\beta^2},\]
	which is not possible with $\alpha<0$ and $\vert x^*\vert <1$, since in that case we would have the contradiction:
	\[J\geq Je^{-\alpha\tau}\geq -\alpha +(1-(x^*)^2+J) \geq 1-(x^*)^2+J,\quad \text{i.e.}\qquad (x^*)^2\geq 1.\]
	Similarly, we can show that no unstable solution exist with $\vert y\vert>\frac 2 3 (J-1)\sqrt{1+2J}$. Indeed, in that case we have $\vert x^*\vert >\sqrt{1+2J}$. Assuming $\alpha>0$, equation~\eqref{eq:CharactRootsFast} implies that necessarily:
	\[-J\leq Je^{-\alpha\tau}\cos(\beta\tau)\leq 1-(x^*)^2+J,\qquad \text{i.e.} \qquad (x^*)^2 \leq {1+2J}\]

In the parameter region not covered by the above characterization, the system may display changes of stability. We now use the characteristic equation~\eqref{eq:CharactRootsFast} to determine parameters corresponding to such changes of stability. We therefore only need to access to the set of parameters for which the characteristic roots of one of the equilibria cross the imaginary axis. There are two such situations: the characteristic root crossing the imaginary axis can either be (i) a real characteristic root (the saddle-node case) or (ii) a pair of purely imaginary characteristic roots (Hopf case).

	Saddle-node bifurcations (corresponding to $\xi=0$) occur if and only if:
	\[x^*=\pm 1\]
	which corresponds, when using the characterization of $x^*$ as a fixed point of the system, to the point in the parameter space:
	\[y=\pm \frac 2 3 \]
	and this bifurcation is independent of the delay $\tau$ and the coupling strength $J$. 

	Hopf bifurcations occur when $\xi=\mathbf{i}\zeta$ for $\zeta>0$, thus necessarily when
	\begin{equation}\label{eq:zetafast} 
		J^2 = (1-(x^*)^2 +J)^2 + \zeta^2 \qquad \textrm{i.e.} \qquad \zeta = \sqrt{(x^*)^2-1}\sqrt{2J+1-(x^*)^2}.
	\end{equation}
	The first observation is that Hopf bifurcations may only arise for $\vert x^*\vert \in (1,\sqrt{1+2J})$, i.e. the only fixed points that may undergo a Hopf bifurcation are $x_{\pm}(y)$. These fixed points lose stability when $\tau$ exceeds $\tau_f^k(x^*)$ given by:
	\begin{equation}
		\label{eq:taufast} \tau_f^k(x^*) = \frac{1}{\zeta}\left(\cos^{-1}\left(\frac{1-(x^*)^2+J}{J}\right)+2k\pi\right).
	\end{equation}
	In order to show that indeed, the system undergoes a generic subcritical Hopf bifurcation at this point, one needs to reduce the system to normal form at this point and compute the first Lyapunov coefficient. Similarly to what we did in order to characterize the criticality of the Hopf bifurcation of the full system in section~\ref{sec:FhNFull} and appendix~\ref{append:NormalForms}, we compute a complex but exact formula for the first Lyapunov coefficient of the fast system aided by a formal computation software (Maple\textregistered \, and Mathematica\textregistered). Numerical evaluation of this coefficient shows that it is always strictly positive (see Figure~\ref{fig:Lyap}), ensuring that the system undergoes a subcritical Hopf bifurcation. 
\end{proof}

\begin{figure}
	\centering
		\includegraphics[width=0.5\textwidth]{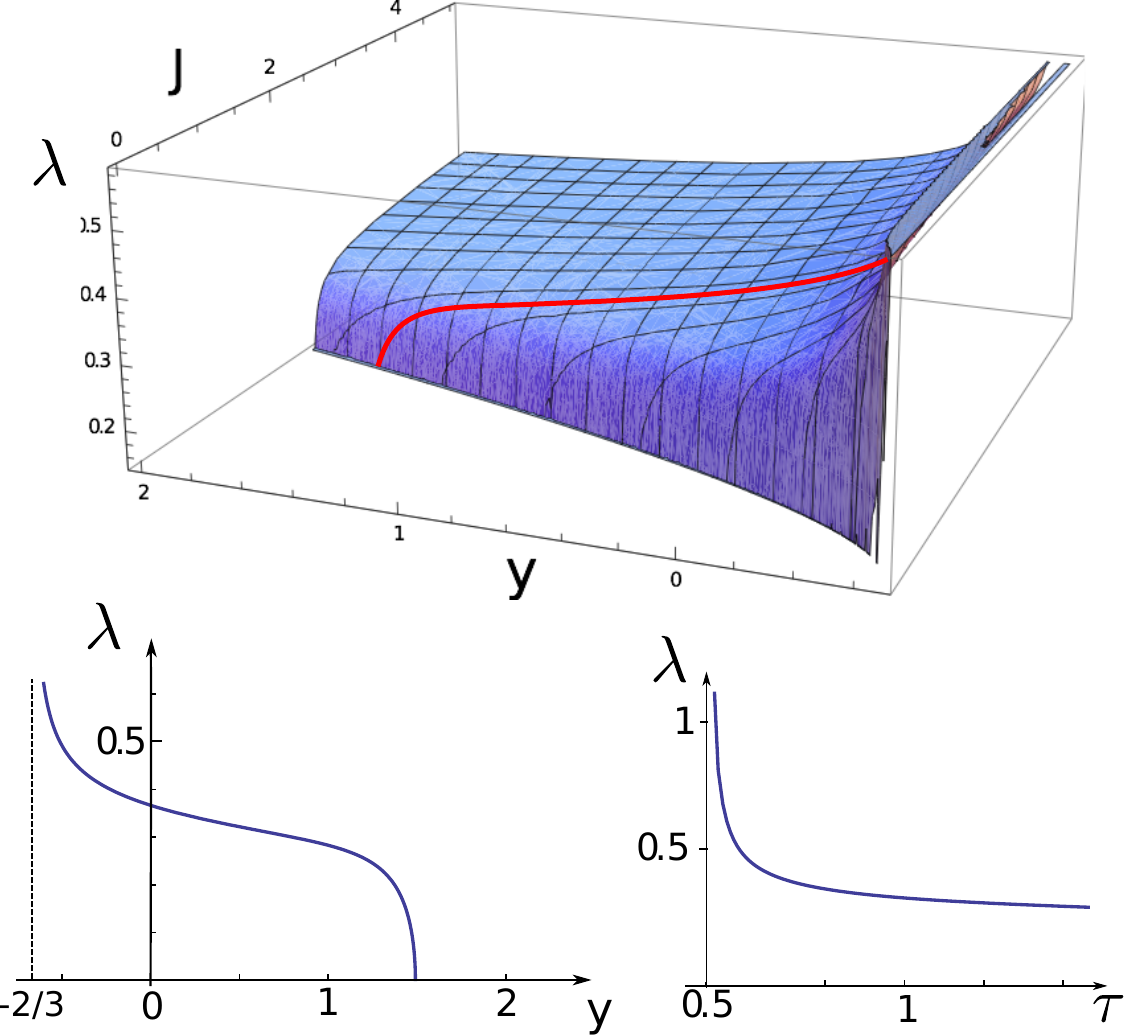}
	\caption{First Lyapunov coefficient of the normal form of the Hopf bifurcation of the fast system. For the whole range of parameters $(J,y)$ considered, the Lyapunov coefficient is strictly positive identifying a generic subcritical Hopf bifurcation. The Lyapunov coefficient diverges when $y\to \pm 2/3$ and $\tau\to 1/J$ and goes to zero at the other end of the Hopf bifurcation curve ($y\to \pm \frac{2}{3}\sqrt{1+2J}(J-1)$ and $\tau\to\infty$). (Bottom) Lyapunov coefficient along the Hopf bifurcation curve for $J=2$ (red curve on the top figure), as a function of $y$ (left) or $\tau$(right).}
	\label{fig:Lyap}
\end{figure}
We now characterize in more detail the shape of the Hopf bifurcation curves. When $y\to \pm \frac{2}{3}\sqrt{1+2J}(J-1)$, the fixed point $x_+(y)$ approaches $\sqrt{1+2J}$ and all branches of Hopf bifurcations escape to infinity, since the Hopf bifurcation curves $\tau_f^k(x^*)$ behave, for $x^*$ approaching $\sqrt{2J+1}$ from the left, as:
\[\tau_f^k(x^*)\sim \frac{(2k+1)\pi}{\sqrt{2J}\sqrt{(2J+1)-(x^*)^2}}.\]
For $x^* \to 1^+$, the behavior of $\tau_f^0$ differs from that of $\tau_f^k$ for $k\in \N\setminus \{0\}$. For $k\neq 0$, the Hopf bifurcation curves $\tau_f^k(x)$ diverge to infinity when $x\to 1^+$. In contrast, the curve of Hopf bifurcations $\tau_f^0$ approaches the curve of saddle node bifurcations tangentially at this point, since using the series expansion of $\cos^{-1}(1+z)$, we have:
\[\tau_f^0(x^*)=\frac 1 J + \frac{1}{12 J^2} (1-(x^*)^2) + \mathcal O ((1-(x^*)^2)^2).\]
This is also the case of the graph of the Hopf bifurcation curve parameterized by $y$: for instance looking at the bifurcation around $x_-(y)$ at $y\to \frac 2 3 ^-$ shows that the curve of Hopf bifurcation meets the line of saddle node bifurcations $y=2/3$ tangentially since it satisfies at this point the expansion:
\[\tau = \frac 1 J - \frac 1 {6J^2}\sqrt{\frac 2 3 -y} +\mathcal O (\sqrt{2/3 -y}).\]
This is consistent with the hypothesis that the system undergoes a generic subcritical Bogdanov-Takens (BT) bifurcation, which we demonstrate in the next section. This BT bifurcation point plays a central role in the dynamics of the full system, organizing the shape and criticality of the Hopf bifurcation of the system as discussed in section~\ref{sec:FhNFull}.

\subsection{Bogdanov-Takens singularity}\label{sec:BT}
The curve of Hopf bifurcations ends at a codimension two point corresponding to $y=2/3$ and $\tau=1/J$. At this point, the saddle-node bifurcation line $y=2/3$ and the Hopf bifurcation line collide. Characterizing this bifurcation will hence allow understanding the transition for $y$ close to one, and we shall now reduce the system to normal form around this point. 
\begin{theorem}\label{thm:BT}
	System \eqref{eq:FastDelayed} undergoes generic subcritical Bogdanov-Takens bifurcation at $y=\pm 2/3$ and $\tau=1/J$.
\end{theorem}
\begin{proof}
	Let us start by noting that at this point the linearized operator corresponds to a the Bogdanov-Takens singularity. Indeed, the linearized system at this point is given by the simple equations:
	\begin{equation}\label{eq:FastDelayedLin}
		u_t'= J(u_t-u_{t-\tau}).
	\end{equation}
	It is very easy to see that the two dimensional space of affine functions $t\mapsto \alpha+t\beta$ is contained in the solution space
	of \eqref{eq:FastDelayedLin} with the action of the RHS of \eqref{eq:FastDelayedLin} given by $\alpha+\beta t \to \beta$.
	In the basis given by the functions $1$ and $t$ this gives the matrix
	\[
	\left (\begin{array}{cc} 0&1\\0&0\end{array}\right ).
	\] 
	Using the dispersion relationship~\eqref{eq:CharactRootsFast} we can draw some conclusions on the bifurcation
	diagram in the plane $(\tau, y)$. First of all the saddle node bifurcation 
	occurs for $y=-2/3$ (or $y=2/3$, case treated in an identical manner).
	Near such a bifurcation point, the fixed point relationship~\eqref{eq:FPFast} ensures that $y=-2/3+((x^*-1)^2)+\mathcal O ((x^*-1)^3)$, and using this expansion and the dispersion relationship \eqref{eq:CharactRootsFast}, one can easily show that the curve of Hopf bifurcations is tangent to the $y=-2/3$
	line, which is consistent with the generic Bogdanov-Takens picture.

	In order to ensure that {the local dynamics of the system at this point is indeed given by the universal unfolding of the Bogdanov-Takens bifurcation}, we reduce the system to normal form at this point. We start by changing variables defining $z= y+2/3$ and $u(t)=x(\tau t)$, so that equation~\eqref{eq:FastDelayed} writes
	\begin{align*}
		\dot{u}&	=\tau J (u_t-u_{t-1}) - \tau u_t^2 + \tau z\\
		&= \delta_u +\nu \delta_u - \tau u^2 + \tau z
	\end{align*}
	with $\delta_u=(u_t-u_{t-1})$ and $\nu=(J\tau-1)$, which is a convenient notation since the singular point considered corresponds to $J\tau=1$. Normal form reduction for such delayed equations was analyzed in~\cite{faria-magalhaes:95}. Applying directly the methodology {presented} in that article, we obtain:
	\[\begin{cases}
		\dot{x_1}= \nu x_1 + x_2 -\frac 2 3 \tau x_1^2 +\tau y\\
		\dot{x_2}=2 \nu x_2 -2\tau x_1^2 + 2 \tau y.
	\end{cases}
	\]
	From these equations, we pursue the reduction to bring the system into canonical form of the Bogdanov-Takens bifurcation as for instance given by~\cite{kuznetsov,guckenheimer-holmes:83}, and obtain, with our notations,
	\[\begin{cases}
		\dot{x_3}=z\\
		\dot{z}= 2\tau y + 2\nu z -2\tau x_3^2 -\frac 4 3 \tau x_3 z.
	\end{cases}\]
	This is the canonical normal form of the subcritical Bogdanov-Takens bifurcation.
\end{proof}
This theorem ensures in particular that beyond the Hopf and saddle-node bifurcations already identified, the system presents also a global bifurcation: a curve of saddle-homoclinic bifurcations which, near the point $J\tau=1$, has the expansion 
\begin{equation}\label{eq:SadHom}
	\{y=2/3-98/25 (J\tau-1)^2/\tau, J\tau>1\}. 
\end{equation}

\section{{Numerical computation of global bifurcations of the fast system}}\label{sec-fastmore}

Analytical methods have allowed characterizing all local codimension {one} bifurcation curves (saddle-node and Hopf bifurcations) and a codimension {two} BT bifurcation. From these bifurcations emerge families of periodic orbits and homoclinic cycles that we characterize numerically in the present section. Universal unfolding of the BT bifurcation {discussed} in section~\ref{sec:BT} provides us with a local equivalent of the curve close to $y=2/3$ and $J\tau=1$, given by the quadratic curve~\eqref{eq:SadHom}. In order to understand the global fast dynamics away from the BT bifurcation, we need to continue curves of homoclinic loops, which is very complex mathematically and {a delicate numerical task}. Here, aided by numerical simulations performed using XPPAut, we obtained a relatively simple bifurcation picture (represented schematically in Figure~\ref{fig:FastBifs}), that may be explained through simple, two-dimensional representations of a virtual phase plane given by $(x_{t},x_{t-\tau})$ (Figure~\ref{fig:Phases}). These phase portraits suggest that the dynamics of~\eqref{eq:FhNdelay} is essentially two-dimensional, which would suggest\footnote{The proof of existence of such a manifold is beyond the scope of this paper.} the existence of an inertial manifold possibly described by the two variables $x(t)$ and $x(t-\tau)$.
\begin{figure}[!h]
	\centering
		\includegraphics[width=.4\textwidth]{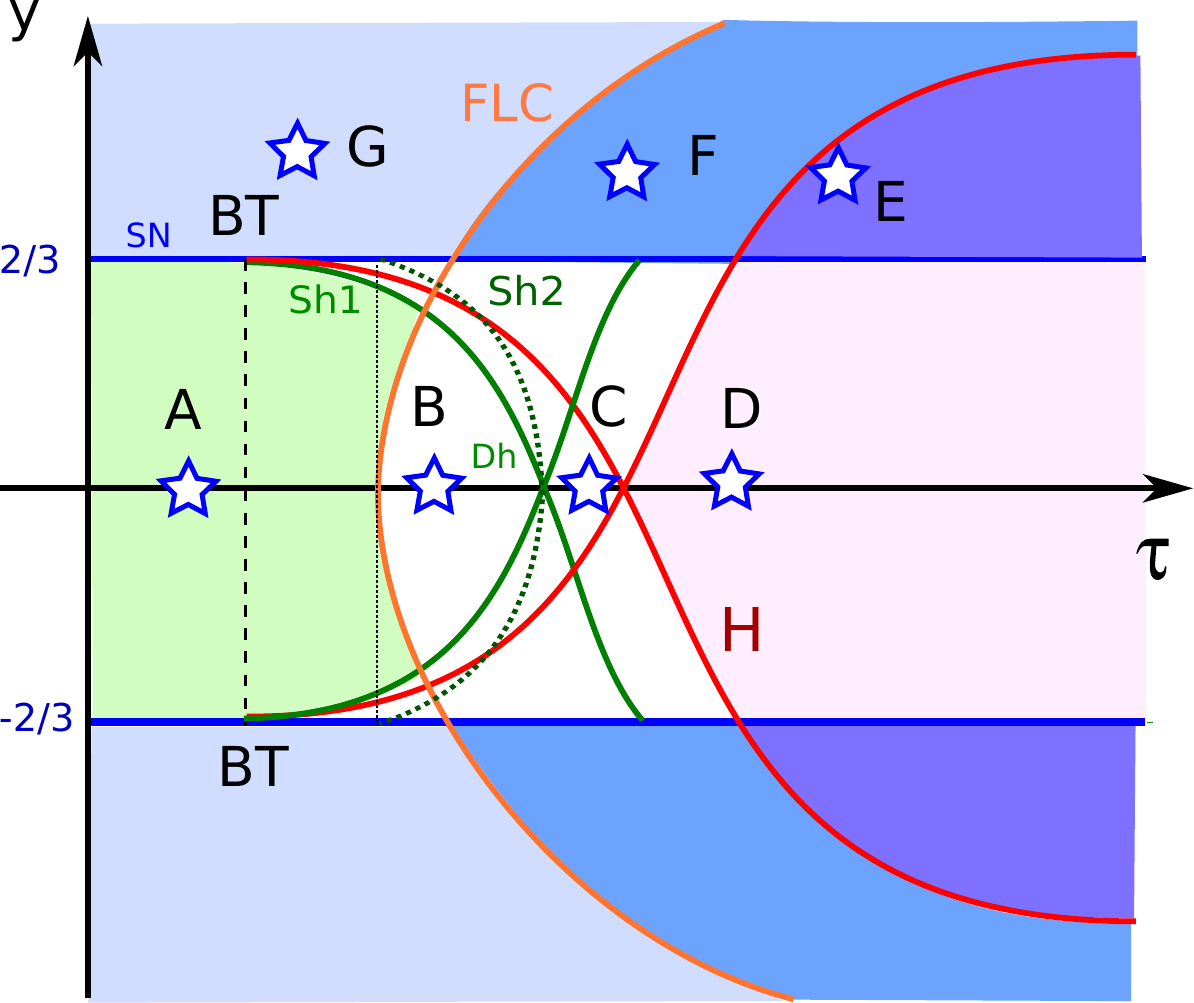}
	\caption{Bifurcation diagram of the fast system in the parameter space $(J,y)$. The diagram is {symmetric} with respect to the reflection $y\mapsto -y$. The unfolding of the Bogdanov-Takens points (BT) at $\tau=1/J$ yields a line of saddle-node bifurcations (SN, blue) at $y=2/3$, a curve of Hopf bifurcations (H, red), together with a saddle-homoclinic bifurcation (Sh1, green) enclosing one fixed point. The fold of limit cycles (FLC) is indicated in orange, and the second line of saddle homoclinics (Sh2) enclosing two fixed points is depicted in dashed green. Saddle homoclinic lines intersect along the line $y=0$, giving rise to a double homoclinic loop (Dh). Seven specific locations are indicated with blue stars, with extended phase portraits provided in Figure~\ref{fig:Phases}.}
	\label{fig:FastBifs}
\end{figure}

{The numerically evidenced bifurcations of cycles of the system, obtained through extensive simulations, are the following:
\begin{itemize}
	\item The saddle-homoclinic bifurcation curves emerge from the BT points and {end} on the opposite manifold of saddle-node bifurcations (solid green curves Sh1). These homoclinic loops enclose only one fixed point. Because of the symmetry of the system $(x,y)\leftrightarrow (-x,-y)$, there exists a double-homoclinic loop at $y=0$ for a given value of $\tau>1/J$ (Dh point).
	\item We numerically {establish} the presence of additional periodic orbits of relatively large amplitude, one stable and enclosing the other one which is unstable. As parameters are varied, we find that the two cycles are connected through a codimension two fold of limit cycles (orange curve in Figure~\ref{fig:FastBifs}).
	\item The smaller and unstable limit cycle described above {terminates}, in a certain range of parameters, on the manifold of saddle fixed point in a saddle-homoclinic loop (dotted green line Sh2 in Figure~\ref{fig:FastBifs}). Contrasting with the curve Sh1 of homoclinic loops, the Sh2 curve correspoonds to homoclinic loops enclosing the two fixed points. Moreover, the curve also intersects the line $y=0$ at the point Dh where the system displays the double-homoclinic loop already mentioned above.
	\item for {the} values of the parameters {for} which the fold of limit cycles {exists} but the Sh2 curve no {longer} exists (for values of $\tau$ larger than the value corresponding to the double homoclinic loop Dh), it actually continues the family of unstable limit cycles emerging from the Hopf bifurcation. 
\end{itemize} }

{Although this description may appear relatively complex, it seems to be constrained to a planar dynamical system and therefore has a fairly simple geometrical interpretation in the virtual extended phase plane $(x_t,x_{t-\tau})$ mentioned above. We now describe, in this space, the {changes} of the dynamics as a function of $\tau$ and $y$.}

\ifthenelse{\boolean{DisplayBigFigs}}{
\begin{figure}[!h]
	\centering                   
		\includegraphics[width=.9\textwidth]{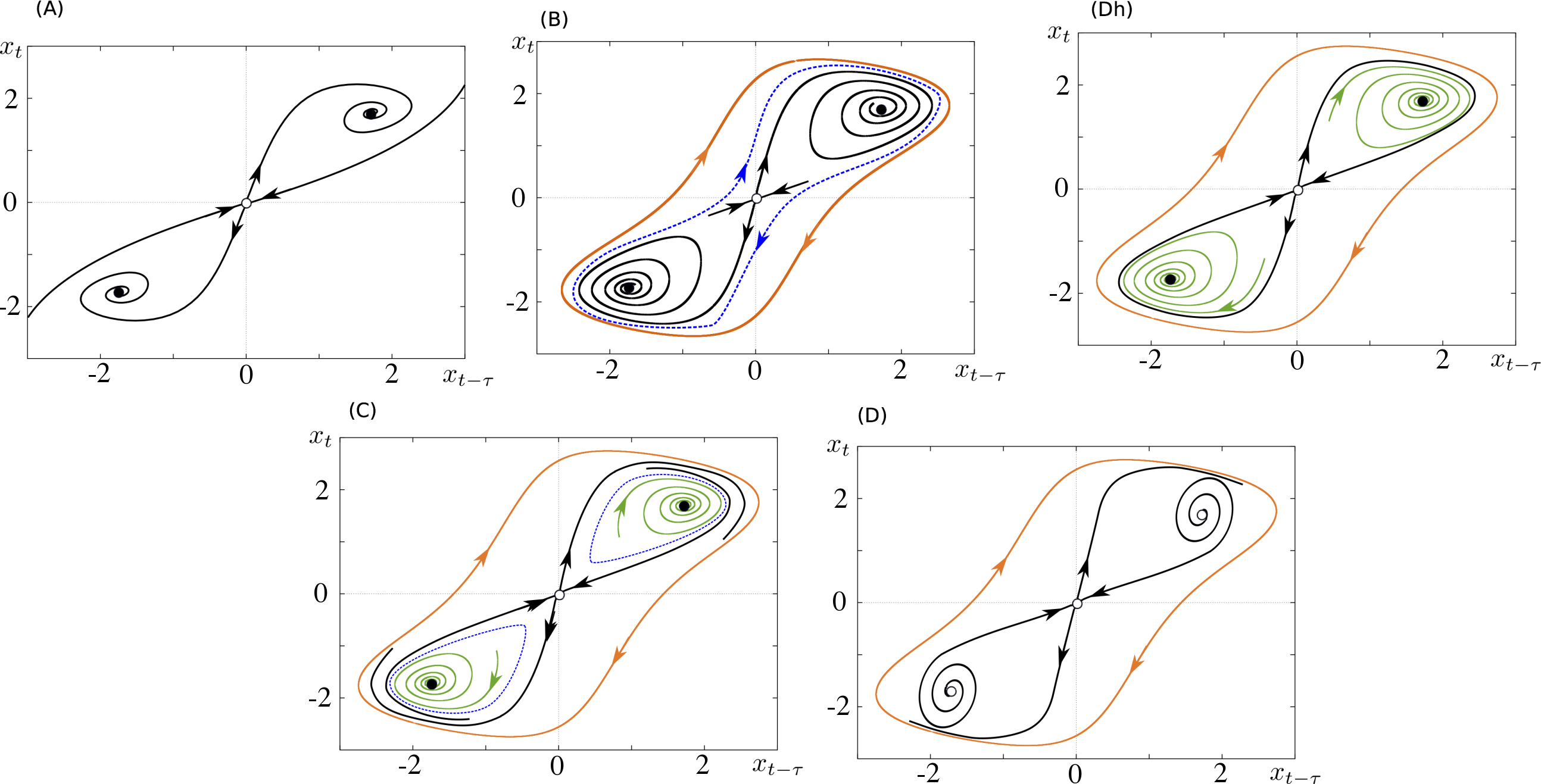}
		\includegraphics[width=.9\textwidth]{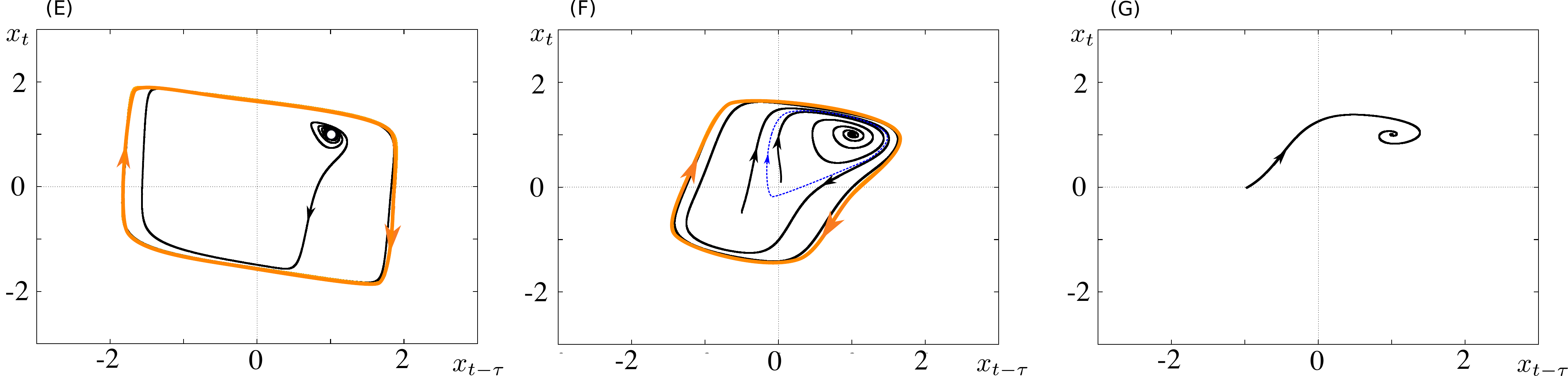}
	\caption{Phase space representation of the fast system in the axes $(x_{t-\tau},x_t)$ for $y=0$ and different values of $\tau$, corresponding to the stars in Figure~\ref{fig:FastBifs} (see text). The trajectories were obtained with XPP Aut. As before, $J=2$. (A-D): $y=0$ and distinct values of $\tau$. (A) $\tau=0.4$, (B) $\tau=0.6$, (Dh) $\tau=0.664$, (C) $\tau=0.7$, (D) $\tau=1$. (E-G) $y=0.8$, (E): $\tau=1.5$, (F) $\tau=0.8$, (G): $\tau=0.5$. All curves are actual solutions of the delayed fast system, except the dashed lines corresponding to the unstable trajectories, namely the unstable limit cycles and the stable manifold of the saddle fixed point. See text for a description of the phase portraits.}
	\label{fig:Phases}
\end{figure}}

{Let us first discuss the symmetric case corresponding to $y=0$, in which case the system has the reflection symmetry $x\leftrightarrow -x$ (see Figure~\ref{fig:Phases}). For small delay, the system has a phase plane similar to the case $\tau=0$: it features two stable foci, a saddle equilibrium and no cycle (Figure~\ref{fig:Phases}(A)). The first bifurcation occurring as delays are increased is the fold of limit cycles (orange curve in Figure~\ref{fig:FastBifs}). At this bifurcation point appear a two cycles, one attractive and one repulsive (Figure~\ref{fig:Phases} (B), orange and blue loops). In this regime, the stable manifold of the saddle fixed point winds around the unstable cycle, while the unstable manifolds converge towards the stable fixed points. As delays are increased, the unstable cycle progressively shrinks and approaches the stable and unstable manifolds of the saddle fixed points. It eventually {becomes} identical to the stable and unstable manifolds precisely when $\tau$ reaches the value corresponding to the double-homoclinic bifucation (Dh point in Figure~\ref{fig:FastBifs}). This loop is also the trace of the symmetric continuation of the saddle-homoclinic loops emerging from the two BT bifurcations, and {at} this point, stable and unstable manifolds of the saddle fixed point coincide (Figure~\ref{fig:Phases} (Dh)). For slightly larger values of the delay, the system shows two unstable periodic orbits around the stable fixed points, towards which the stable manifold of the saddle fixed point {converges}, {while} the unstable manifold converges towards the large stable cycle (Figure~\ref{fig:Phases} (C)). The unstable orbits cycling around the stable fixed points shrink, until $\tau$ reaches the value of the Hopf bifurcation (arising simultaneously {at two fixed points} due to symmetry). At {the Hopf point} the stable fixed {points} become unstable {foci} and the unstable cycles disappear. The unstable manifold of the saddle fixed point keeps converging towards the stable relaxation cycle, and the stable manifold winds around the unstable {foci} (Figure~\ref{fig:Phases} (D)). }
\ifthenelse{\boolean{DisplayBigFigs}}{
\begin{figure}[!h]
	\centering
		\includegraphics[width=.7\textwidth]{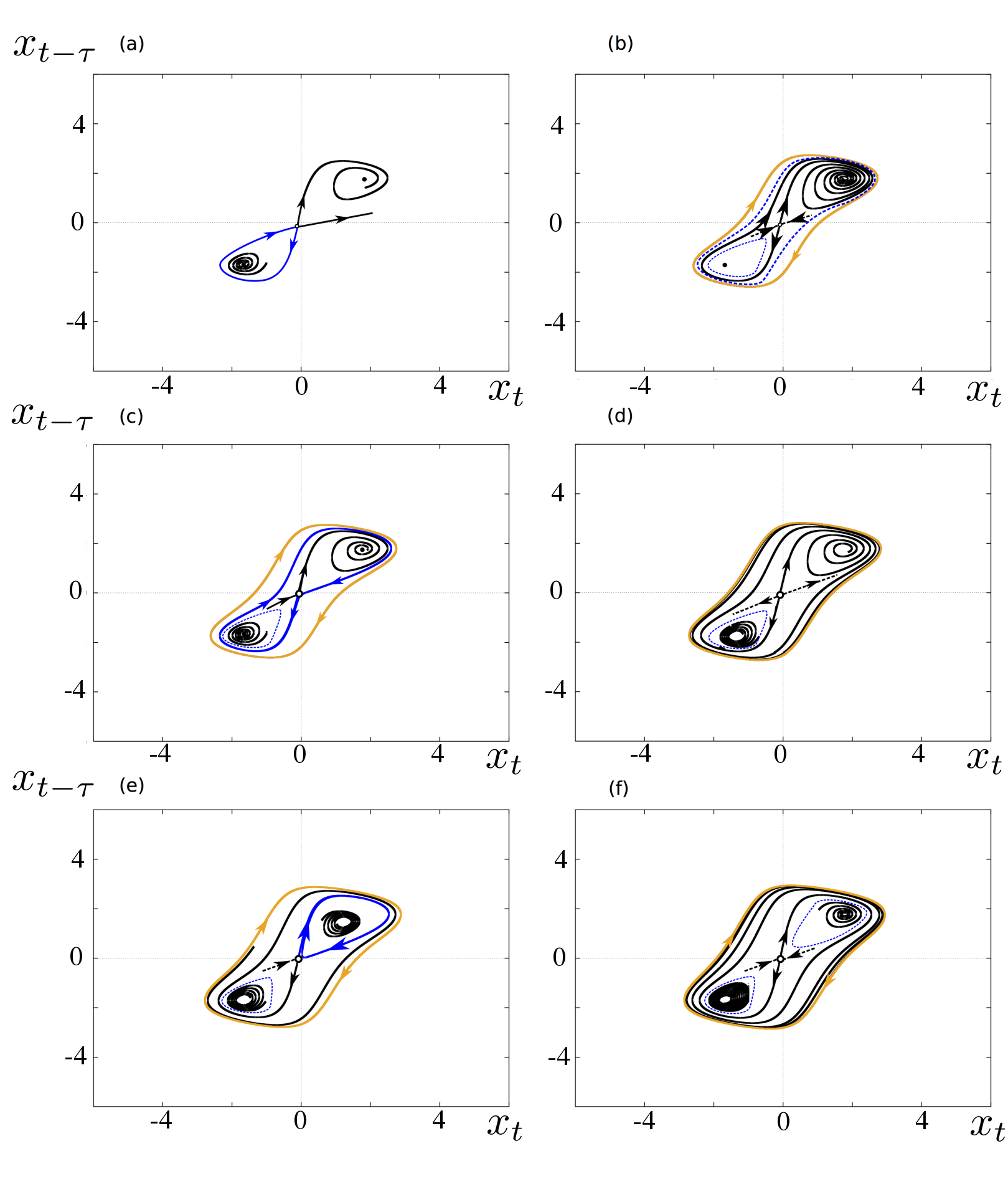}
	\caption{Phase space representation of the fast system in the axes $(x_{t-\tau},x_t)$ for $y=0.1$ and different values of $\tau$. The trajectories were obtained with XPP Aut. 
	As before, $J=2$. (a) $\tau=0.6449$, (b) $\tau=0.652$, (c) $\tau=0.656$, (d) $\tau=0.67$, (e) $\tau=0.6832$, (f) $\tau=0.7$. See text for a description.}
	\label{fig:AppendixPhase}
\end{figure}
}

Asymmetric situations arise when $y\neq 0$. When $\vert y \vert>2/3$, the system has a unique fixed point, which cannot be unstable if $\vert y\vert >\sqrt{1+2J}$. The system can feature an additional pair of limit cycles (one stable and one unstable). A typical case for $2/3 < y < \sqrt{1+2J}$ is depicted in Figure~\ref{fig:Phases}(E-G). Let us {finally} discuss the {what happens if} $0<y<2/3$. In {this} case homoclinic bifurcations do not arise simultaneously at a double homoclinic loop, but sequentially. An example of phase portraits is provided in Figure~\ref{fig:AppendixPhase} for $y=0.1$. In {this} case, the first homoclinic bifurcation (Sh1) already arises before the fold of limit cycles. A similar situation as in Figure~\ref{fig:Phases}(A) occurs for small delays, and a homoclinic loop arises (Figure~\ref{fig:AppendixPhase}(a)), {giving rise to} an unstable periodic orbit enclosing the smallest fixed point (Figure~\ref{fig:AppendixPhase}(b)). After the fold of limit cycles, a pair of stable and unstable cycles emerges (Figure~\ref{fig:AppendixPhase}(b)). The unstable cycle progressively shrinks towards the saddle fixed point, until reaching the homoclinic bifurcation Sh2 (Figure~\ref{fig:AppendixPhase}(c)). At this point, as delays are further increased, the unstable loop disappears, and the system is left with an unstable orbit enclosing the smallest fixed point, a stable fixed point, a saddle and a stable limit cycle (Figure~\ref{fig:AppendixPhase}(d)). Further increasing the delays leads to crossing the saddle homoclinic line that emerges from the BT point at $y=-2/3$: a homoclinic loop encloses the {the fixed point with the largest value of $x$} (Figure~\ref{fig:AppendixPhase}(e)). For larger delays, the loop becomes an unstable periodic orbit enclosing the {the fixed point with the largest value of $x$} (Figure~\ref{fig:AppendixPhase}(f)), and we are in a situation similar to Figure~\ref{fig:Phases}(C). As delays are further increased, the {the fixed point with the least value of $x$} will undergo a Hopf bifurcation and lose stability as the loop enclosing it collapses (not shown). The phase portrait is similar to Figure~\ref{fig:AppendixPhase}(e), except that the loop around the {the fixed point with the least value of $x$} disappeared and that fixed point lost stability. The stable manifold of the saddle fixed point that converged towards this cycle now winds around this stable fixed point. 

\bigskip

Based on these descriptions of the fast dynamics, we are now in a position to account for the complex oscillatory patterns observed in our delayed slow-fast equation. 

\section{Complex dynamics of the delayed FitzHugh-Nagumo system}\label{sec:Complex}
\revision{Now that we characterized the dynamics of the fast system, we investigate the dynamics of the full system. The origin of small oscillations arising in the regime $J\tau<1$ is related to the presence of canard cycles and is investigated more in detail in~\cite{krupa-touboul:14a}. In that paper, as an illustration for our theory on canard explosions in delayed differential equations, we have shown that the FitzHugh-Nagumo system~\eqref{eq:FhNdelay} undergoes a {super}critical canard explosion along the branch of folds $y=\pm 2/3$ for $J\tau <1/2$, and we conjectured that it persists for $J\tau<1$. We refer to that paper for the explanations of the emergence of small oscillations for $a$ close to $1$, and their instability as delays are increased (Figure~\ref{fig:AllBehaviors}-\emph{Small Oscillations}).

Here, we focus on the regimes $J\tau>1$, in which the dynamics of the fast system is nontrivially governed by the delays. The combination of the slow (quasi-static) {evolution} of the variable $y$ as the fast system evolves according to the above descriptions will provide an explanation for the different behaviors observed. The descriptions provided below are all valid for $\eps$ sufficiently small.

For any $J$ and $\tau$, when $a\leq\sqrt{1+2J}$, the system converges towards the stable fixed point $(a,a-a^3/3)$. This is also the case for $a<\sqrt{1+2J}$ when $\tau<\tau_f^0(J,a)$ (these provide two boundaries, for any fixed $\tau$, delineating the region where solutions are fixed points, or `stationary' yellow region, in Figures~\ref{fig:MMOPhase} and~\ref{fig:BurstExplained}).

But if $a<1$ or if $1<a<\sqrt{1+2J}$ and $\tau>\tau_f^0(J,a)$, the full system does not have any stable equilibrium, and non-stationary behaviors therefore emerge. These can be of distinct types depending on the shape of the bifurcation diagram of the fast system as $y$ varies. We describe these in more detail below. 
}

\subsection{Relaxation and Mixed Mode Oscillations}\label{sec:MMOs}
For $\tau$ slightly larger than $1/J$, we have observed that the fast system shows a subcritical Hopf bifurcation, associated with a family of limit cycles undergoing a saddle-homoclinic bifurcation, and no stable cycle is present in the fast system (as long as we are below the value corresponding to the fold of limit cycles). We display in Figure~\ref{fig:MMOPhase} the typical bifurcation diagram of the fast system in that regime as a function of $y$. For $\eps>0$ small enough, $y$ varies slowly in time and the trajectories of the full system can be understood as slow {motions along}  the bifurcation diagram of the fast system.

In that regime, depending on $a$, the system either shows fixed-point behaviors (for $a$ within the yellow region of Figure~\ref{fig:MMOPhase}) or MMOs (for $a$ within the pink region of Figure~\ref{fig:MMOPhase}). 
\begin{figure}[h]
	\centering
		\includegraphics[width=.3\textwidth]{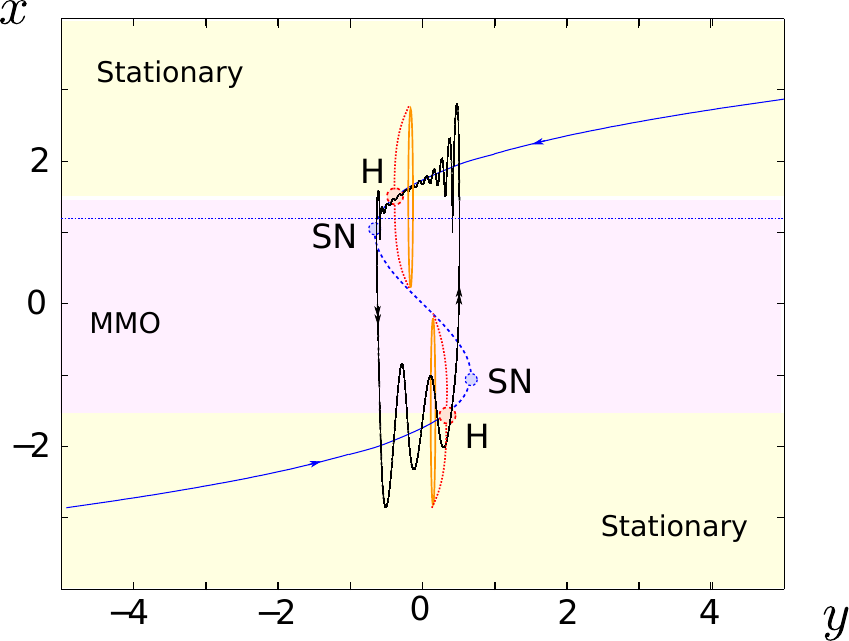}\qquad 
		\includegraphics[width=.3\textwidth]{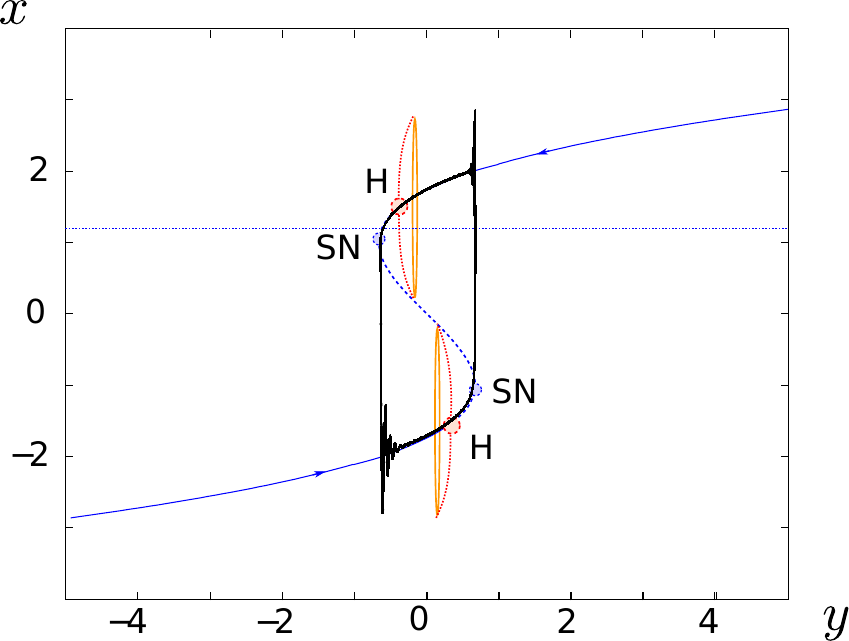}
	\caption{The MMO cycle: $J=2$ and $\tau=0.65$. The bifurcation diagram of the fast system is superimposed with a trajectory of the system in the plane $(x,y)$. Blue lines are fixed points (solid: stable, dashed: unstable). SN: saddle-node, H: Hopf, red dashed lines: unstable periodic orbits, and orange cycle is the saddle-homoclinic cycle. Dashed horizontal line corresponds to the value of $a$. (a) $\eps=0.05$, (b) $\eps=0.01$. Yellow and pink regions correspond to choices for $a$ leading respectively to stationary or periodic (relaxation oscillations or MMO) behaviors. }
	\label{fig:MMOPhase}
\end{figure}

Let us consider $a$ within the pink region. Since the fast system does not have any {periodic} orbit, for almost all initial {conditions} of the full system the fast variable converges rapidly towards {a} stable branch of the slow manifold (blue line in Figure~\ref{fig:MMOPhase}). From that point, as time goes by, the value of the slow variable starts evolving according to the {slow} dynamics, along the arrows plotted on the slow manifold. Assume for instance that it converges {to} the branch of the slow manifold {corresponding to} $x<a$. Then $y$ slowly increases and the variable $x$ drifts accordingly along the slow manifold, and will necessarily reach the {subcritical} Hopf bifurcation point where the branch of the stable manifold loses stability. At this point, the fast system cannot stay in the vicinity of this branch of the stable manifold and eventually jumps towards the upper branch of the slow manifold. Because of the delays, the fast variable makes an overshoot, and subsequently displays fast damped oscillations while converging towards the critical manifold (negative delayed feedback). {On this branch $x>a$, hence} the variable $y$ {slowly decreases, crosses} the Hopf bifurcation on the upper branch and jumps towards the lower branch of fixed points. This dynamics creates a periodic solution showing small oscillations interspersed with large amplitude relaxation oscillations, the Mixed-Mode Oscillations shown in Figure~\ref{fig:AllBehaviors}. In the phase portraits of the fast system pictures, the system slowly switches back and forth between the situation described in Figure~\ref{fig:Phases}(G) and the symmetric situation. Note that for $a\neq 0$, the oscillation is asymmetric: for $a>0$ for instance, the variation of $y$ is much faster on the lower branch than on the upper branch. Therefore, on the upper branch, the trajectory gets closer of the slow manifold than on the lower branch, which is highly visible for larger values of $\eps$ (left panel of Fig~\ref{fig:MMOPhase} for $\eps=0.05$) compared to cases where $\eps$ is smaller (right panel, $\eps=0.01$). Note that in the latter case, we also see a delayed Hopf bifurcation phenomenon: the trajectories closely follow the unstable branch of the slow manifold. In the former case, an additional phenomenon may be noted: as the trajectories lose stability on the upper branch, the system shows small oscillations, signature of the emergence of complex eigenvalues along the critical manifold and small oscillations are visible. 

\begin{rem}\label{rem-MMOs}
The MMOs described in this section are different than canard mediated MMOs
arising near folded singularities \cite{brons-krupa-wechselberger:06,desroches:12}, for which transitions between regions
of different numbers of small oscillations occur by means of a passage through a canard solution (similar to canard explosion).
MMOs of this type require the presence of two slow variables and a fold singularity. Interestingly, canard transitions 
seem to also play a role in the BT mediated MMOs discussed in this paper. 
Typical trajectories pointing towards the presence of this phenomenon can be computed in the delayed FhN system. Such a trajectory
is depicted in Figure~\ref{fig:AllBehaviors}-\emph{(MMOs)} (left side). 
\end{rem}

\subsection{Bursting and Fast Oscillations}\label{sec:periodicAveraging}\label{sec:Bursting}
As $\tau$ is increased, the bifurcation diagram of Figure~\ref{fig:FastBifs} indicates that stable periodic orbits arise. These will be associated to faster oscillations of the full system, either in the form of a fast oscillation or a burst. We focus here on the case where the fast system shows two folds of limit cycles connecting the branches of stable and unstable periodic orbits. This situation corresponds to the case where $\tau$ is large enough so that the homoclinic bifurcations have disappeared (see Figure~\ref{fig:FastBifs}). This is the case for instance for $J=2$ and $\tau=1$, as we depicted in the bifurcation diagram of Figure~\ref{fig:BurstExplained} computed with the Matlab package DDE Biftool~\cite{DDEBif1,DDEBif2}. In {this} diagram, {shown in in Figure~\ref{fig:BurstExplained}(b)} we display the minimum and maximum amplitude of the periodic orbits of the system (red), while in Figure~\ref{fig:BurstExplained}(a) we replaced the amplitude of the cycle by the average value of the fast variable over a period of the cycle (green line). This quantity is relevant in order to understand the behavior of slow variable when $\eps>0$ small. 

\begin{figure}[!h]
	\centering
		\subfigure[Abstract view]{\includegraphics[width=.3\textwidth]{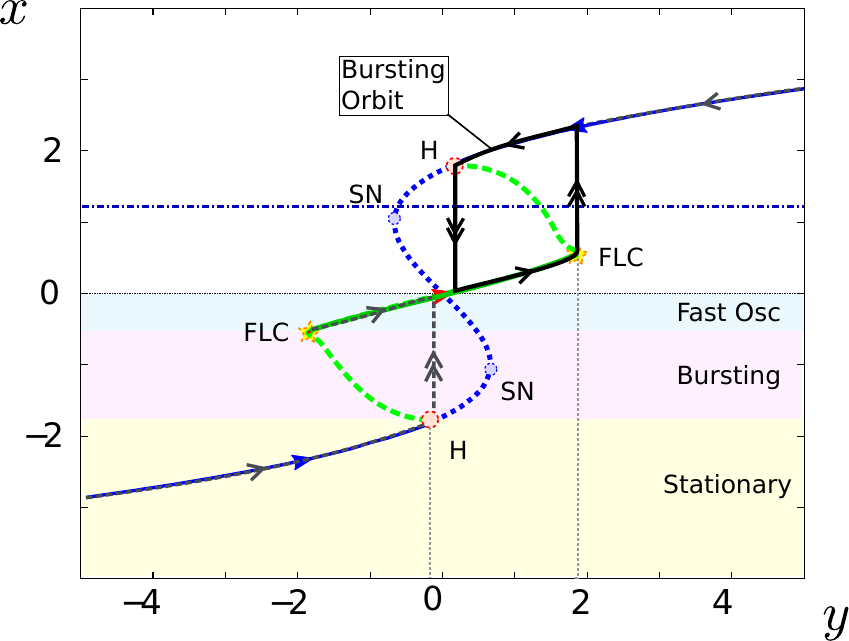}}\qquad
		\subfigure[Actual trajectory, $\varepsilon=0.1$, $a=1.01$]{\includegraphics[width=.28\textwidth]{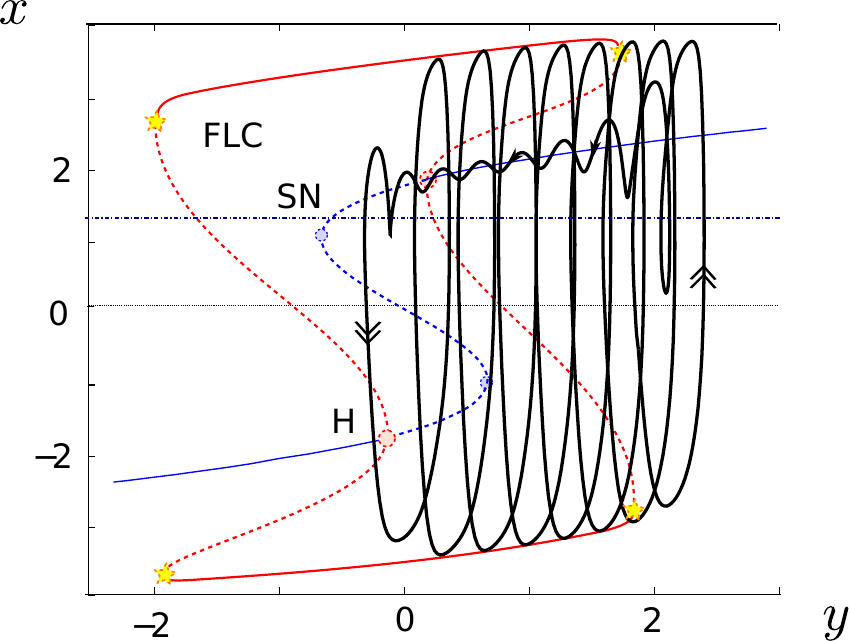}}
		\subfigure[Actual trajectory, $\varepsilon=0.1$, $a=0$]{\includegraphics[width=.43\textwidth]{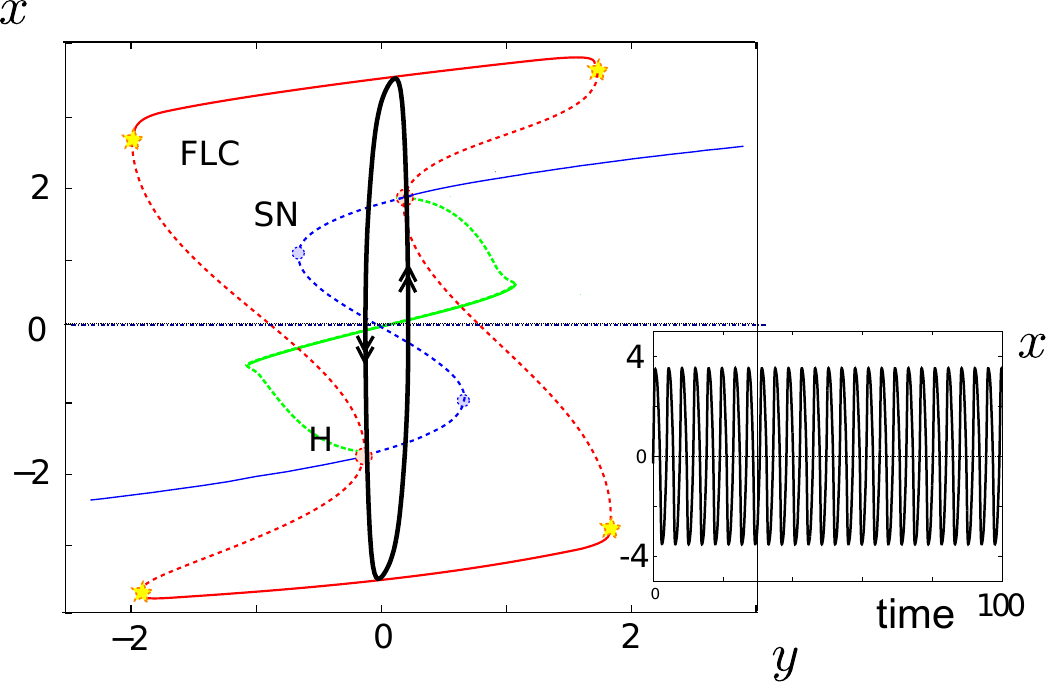}}
		\subfigure[Actual trajectory, $\varepsilon=0.1$, $a=0.5218$]{\includegraphics[width=.43\textwidth]{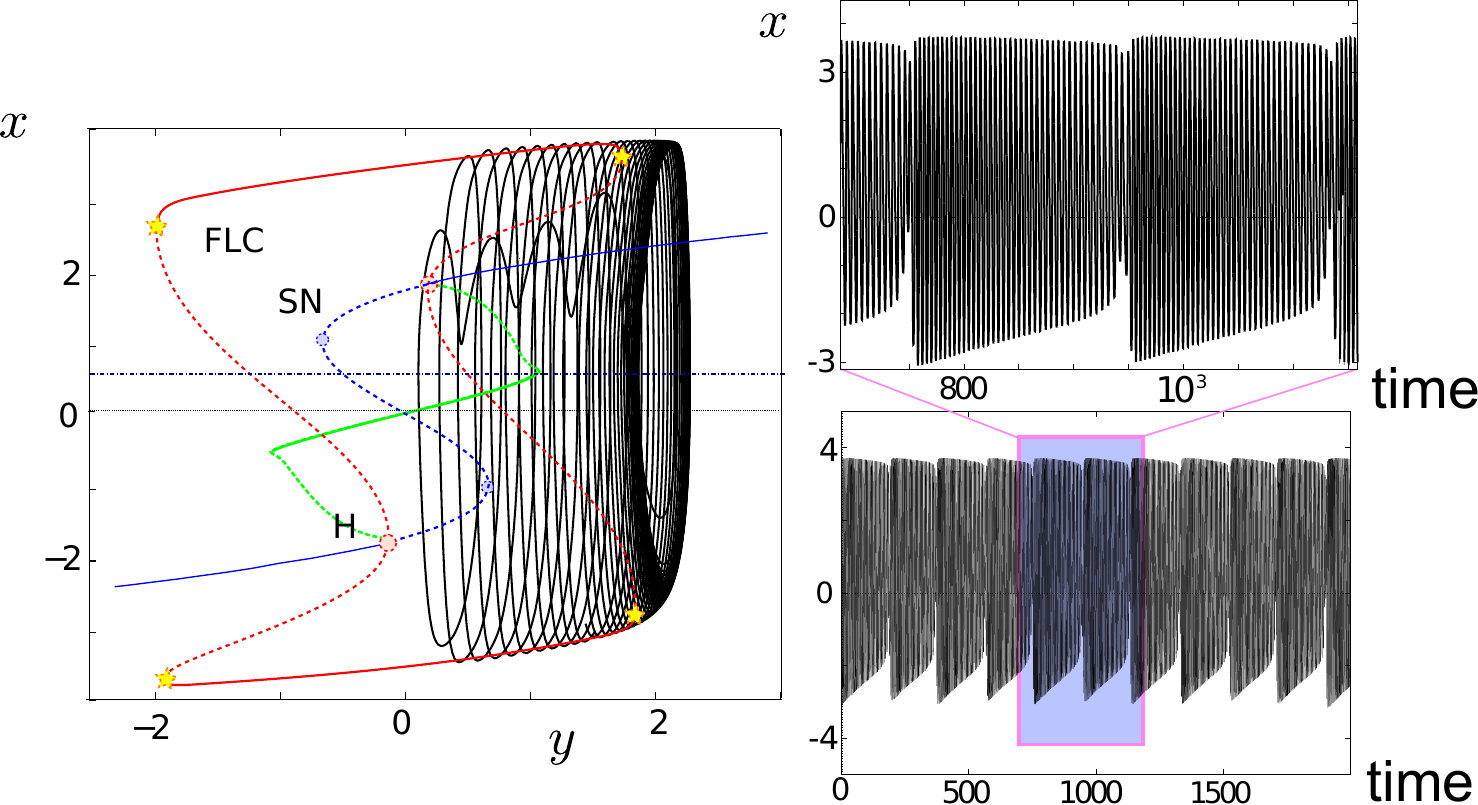}}
	\caption{FhN system, $\tau=1$. (a):Bifurcation diagram of the fast system as a function of $y$. Fixed points are in blue, average amplitude of cycles in green (plain line: stable, dashed: unstable). Regions of fast oscillations, bursting and stationary solutions are depicted for $a<0$ (to be completed by symmetry in the region $a>0$). An example of bursting cycle for $a=1.1$ in the singular limit is shown: the figure indicates the slow dynamics along these manifolds, the bursting orbit (black thick line) which is globally attractive (three possible transient in dashed gray). (b): an actual bursting orbit in the non-zero $\varepsilon$ case. For legibility, we chose $\varepsilon$ relatively large $\varepsilon=0.1$, and this explains why the system slightly departs from bursting cycle depicted in subfigure. (a). Smaller $\varepsilon$ more closely match the bursting cycle, but the number of oscillations in the burst dramatically increase (as $1/\varepsilon$) affecting legibility. (c) a fast oscillation case: $a=0$. }
	\label{fig:BurstExplained}
\end{figure}

Indeed, let us denote {by} $x_w(t)$ the fast system dynamics for a fixed value of the slow variable $y=w$ (which is a parameter of the fast system), and let us virtually uncouple the slow and fast equation.  Integrating the slow dynamics for $x(t)=x_w(t)$ yields:
\[y(t)=y(0)+\varepsilon\,t\,(a-\frac 1 t\int_0^t x_w(s)\,ds)\]
which in the slow timescale $\theta=\varepsilon t$, reads:
\begin{equation}\label{eq:slowdynamics}
	y(\theta)=y(0)+\theta \Big(a-\frac \varepsilon \theta \int_0^{\frac \theta \varepsilon } x_w(s)\,ds\Big)
\end{equation}
We have shown that $x_w(t)$ is one of two types: either a fixed point or a periodic orbit. In the limit $\eps\to 0$, the slow dynamics is therefore written, using equation~\eqref{eq:slowdynamics}, as an implicit system depending on the temporal average of the fast variable 
\[m(w,\tau)=\lim_{t\to\infty} \frac 1 t \int_{0}^t x_{w}(s)\,ds.\] 
In the case of a stationary behavior of the fast subsystem, this quantity is precisely equal to the value of the fixed point, and in the case of a periodic activity of period $T$, this quantity is equal to the average of the cycle $\frac 1 T \int_0^T x_{w}(s)\,ds$. Therefore, in the slow timescale, introducing the delay has the effect of modifying the slow manifold by adding to the stationary solution manifold a branch related to stable periodic orbits of the fast system. 

{From this point of view}, the slow manifold can be seen as the union of the critical manifold and the average value of $x$ {along} the cycles. It is depicted in Figure~\ref{fig:BurstExplained}. It present folds that are associated to the saddle-node bifurcations and to the fold of limit cycles of the fast system (see section~\ref{sec:fastHopf}). Along the branch of fixed points, we also find changes of stability due to the presence of Hopf bifurcations. 

On that manifold, the slow variable $y(\theta)$ evolves depending on the relative position of the manifold compared to the parameter $a$: $y$ is increasing when $m(y(\theta),\tau)>a$ and decreasing otherwise. When the slow dynamics reaches one of these extremal {values} of the slow attractive manifold, a fast switch will occur, depicted in this figure by {the} black vertical line with double arrows.

This analysis allows characterizing the behavior of the system as a function of $a$, and reveals new behaviors that were not observed in Figure~\ref{fig:AllBehaviors}. If the value of $a$ intersects a stable branch of the limit cycles manifold (light blue region of Figure~\ref{fig:BurstExplained}), the value of $y$ will stabilize at $a$ and the system will display fast oscillations (see~\ref{fig:BurstExplained}(c)). For larger $a$ (in absolute value), within the yellow region, the line $x=a$ intersects the critical manifold at a stable fixed point of the fast system, and the full system will stabilize at a fixed point. In contrast, in an intermediate region of values of $a$ (pink region), the line $x=a$ neither intersects the stable branch of average value of limit cycle nor a stable branch of the critical manifold. In that case, as plotted schematically in Figure~\ref{fig:BurstExplained}(a), the system shows a relaxation-like cycle including a branch of the fast oscillations manifold: the system switches periodically between the periodic orbits manifold the fast system and the branch of fixed point of the slow manifold, thus producing bursts:
\begin{itemize}
	\item when the fast system stabilizes on the branch of stable limit cycles, the system presents fast oscillations in the fast variable, and the slow variable slowly increases (with a speed {bounded below})
	\item  this persists until the slow variable $y$ reaches the value associated to the fold of limit cycles of the fast subsystem. Subsequently, the system jumps on {to} the slow manifold associated to stable fixed point of the fast dynamics
	\item On that manifold, the slow variable $y$ decreases with a lowerbounded speed, and therefore will reach the Hopf bifurcation, subsequently leading the trajectory to leave the stable manifold and jump onto the manifold of stable oscillations. 
\end{itemize}
This scenario hence occurs periodically in time, defining a periodic orbit composed of a phase of fast oscillations followed by a silent period, in other words a burst. A trajectory of the dynamical system superimposed to the slow manifold is plotted in Figure Figure~\ref{fig:BurstExplained}(b). In the phase portraits of the fast system pictures, the system slowly switches back and forth between the situation described in Figure~\ref{fig:AppendixPhase}(a) to that of Figure~\ref{fig:Phases}(E). 
Eventually, we note that for values of $a$ close to the transition from fast spiking to bursting, complex oscillatory patterns made of fast oscillations with modulated amplitude are found (see~\ref{fig:BurstExplained}(d)) evocative of torus canard phenomena~\cite{burke2012showcase}.

\subsection{Chaotic transition}\label{sec:ChaosTransition}
We now investigate the origin of the chaotic orbits displayed in Figure~\ref{fig:AllBehaviors}\emph{(Chaos)}. These solutions, arising for a range of values of the delay between those related to MMOs and those related to the presence of stable bursting cycles, display irregular alternations of MMO-like or bursting orbits, together with mixed trajectories. {More precisely}, the chaotic orbits observed (Figure~~\ref{fig:AllBehaviors}\emph{(Chaos)}) show the presence of relaxation cycles of relatively regular period, but during the increase {of the} phase of the slow variable, the fast variable shows chaotic alternations between at least 3 behaviors: (i) damped oscillations convergence towards the stable manifold (MMO-like trajectory with no fast oscillation), (ii) large periodic oscillations (bursting-like trajectory) or (iii) a mixture between the two types of orbits.

In order to understand the origin of this chaotic switching, one needs to understand the fate of the trajectories leaving the upper branch of the stable manifold (in the case $a>0$ that we have been considering thus far). {Such trajectories} follow the {slow} manifold {up to} the vicinity of the Hopf bifurcation point where the critical manifold loses stability. Trajectories either switch to the other branch of fixed points (MMO case) or to the periodic orbit (busting case). This choice depends on the possible {fast} trajectories emerging from the neighborhood of the fixed point as it loses stability. In the MMO case described in section~\ref{sec:MMOs}, there is no periodic orbit at all and the system goes towards the stable fixed point, and in the bursting case described in section~\ref{sec:Bursting}, in a wide neighborhood of the Hopf bifurcation, the only stable orbit is the fast periodic orbit. The chaotic phase occurs when both attractors are stable in the vicinity of the Hopf bifurcation, and moreover when the separation between the two {behaviors} is very sensitive. This is the case when the unstable orbit disappeared but not the large orbit, which is precisely what we observe in Figure~\ref{fig:PhaseSpaceChaos}. In that diagram we show the phase plane representation of the fast system (in the coordinates $(x(t),x(t-\tau))$) for the value of $\tau$ corresponding to the chaotic region in Figure~\ref{fig:AllBehaviors}\emph{(Chaos)}. We observe indeed that in {this} case, the system can either switch towards the stable fixed point or to the stable periodic orbit: one branch the unstable manifold of the saddle fixed point converges on one side to the fixed point and the other branch to the periodic orbit, while both branches of the stable manifold wind around the unstable fixed point. The system is therefore extremely sensitive {to} the precise location of the trajectory as it leaves the neighborhood of the now unstable fixed point: in the two-dimensional extended phase space metaphor of the system, the two branches of stable manifold of the saddle do separate those trajectories converging towards the fixed point and those converging towards the periodic orbit. The attraction bassins of the fixed point and cycles are interwoven in a very intricate fashion\footnote{In the vicinity of the fixed point after the Hopf bifurcation, it is clear that the expansion rate in the delayed system becomes very large because of the separation of trajectories between those converging to fixed points and to the periodic orbit}. We interpret this mechanism as the origin of the emergence of the irregular alternation of MMO-like and burst-like trajectories in this parameter regime.
\begin{figure}
	\centering
		\subfigure[Phase Portrait]{\includegraphics[width=.35\textwidth]{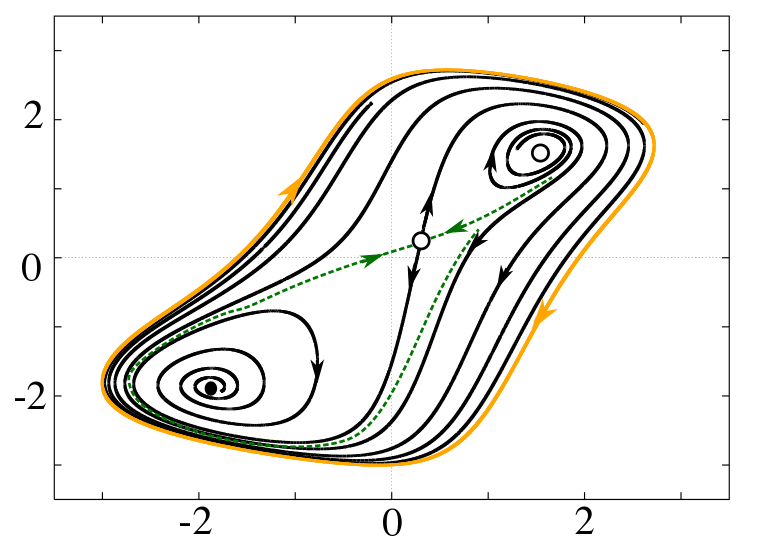}\label{fig:PhaseSpaceChaos}}\qquad 
		\subfigure[Fast system bifurcations]{\includegraphics[width=.32\textwidth]{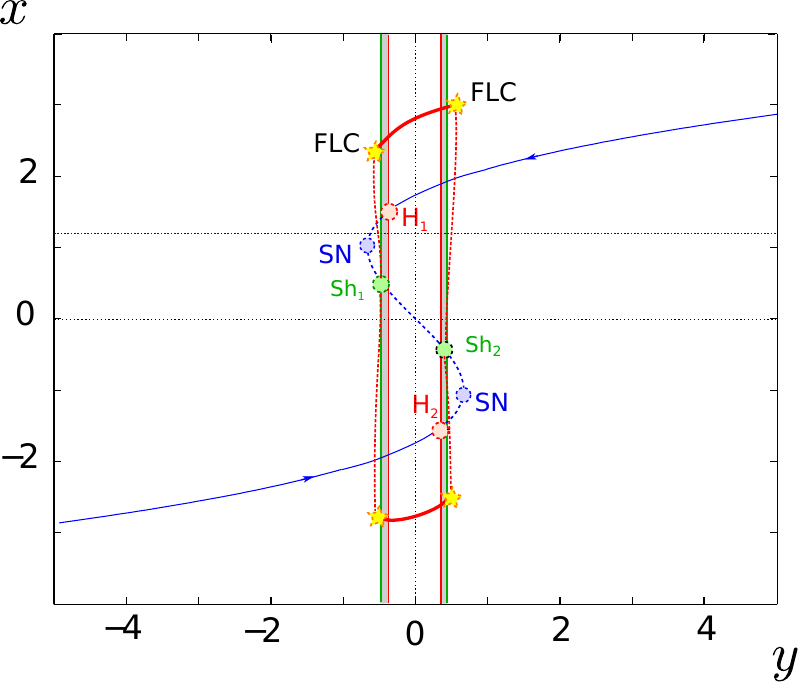}}\\
		\subfigure[MMO cycle]{\includegraphics[width=.32\textwidth]{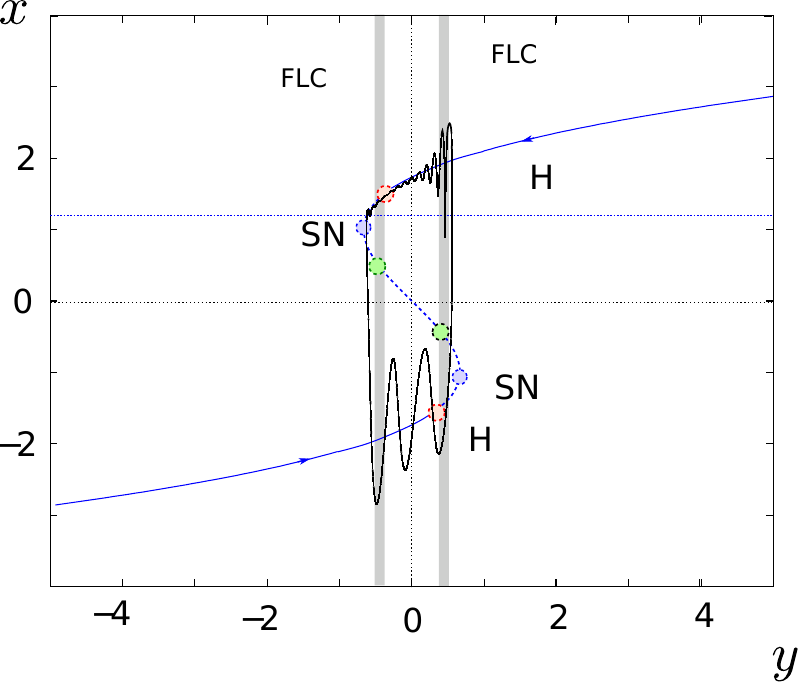}}
		\subfigure[Bursting cycle]{\includegraphics[width=.32\textwidth]{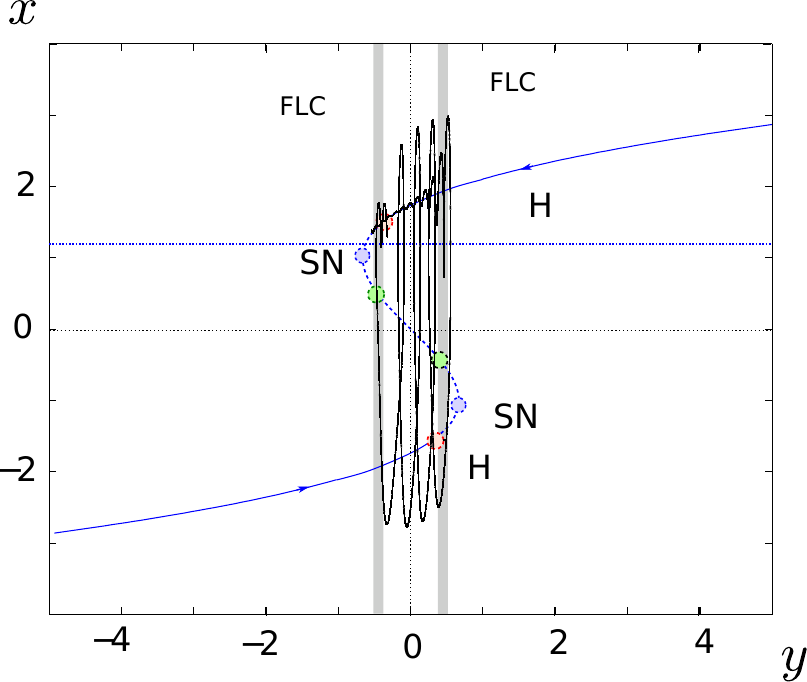}}
		\subfigure[Mixed cycle]{\includegraphics[width=.32\textwidth]{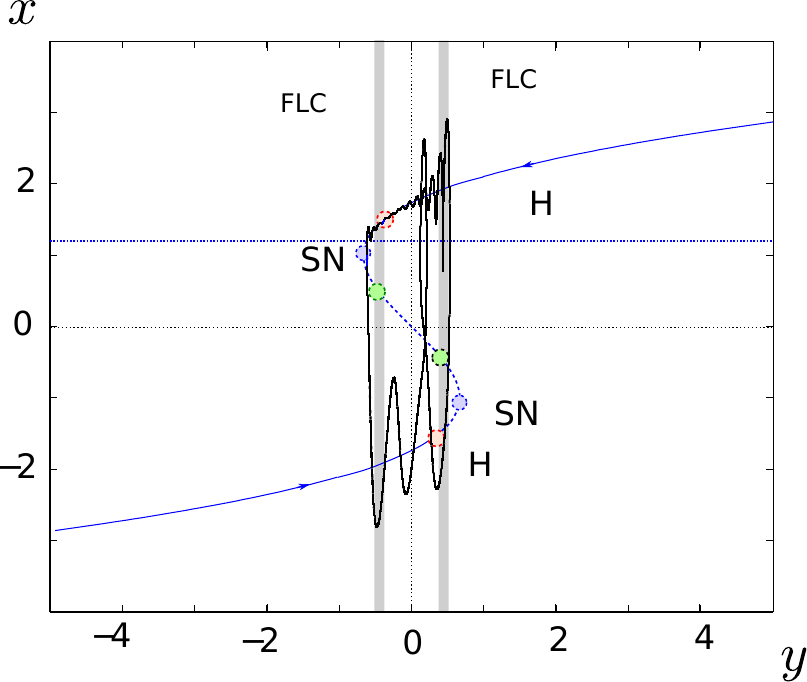}}
	\caption{Chaotic orbits. $\tau=0.7$, $\eps=0.05$. (a) Phase plane in a chaotic situation ($y=-0.4$) the stable manifold of the saddle fixed point winds around the unstable orbit. (b) Bifurcation diagram of the system as a function of $y$. The saddle homoclinic bifurcations (Sh2) and the Hopf bifurcations (H) are such that there exists a range of values of $y$ (gray region) in which the phase portrait is of type (a): no unstable cycle acts as a separatrix between the fast periodic orbit and the fixed point. (c,d,e) represent three trajectories for $\eps=0.05$ (large for legibility). (c) MMO-like trajectory (delayed switch favors such trajectories), (d) bursting-like trajectory (favored by early switches), (e) is a mixed trajectory: switch from the critical manifold at an intermediate location, yielding loose convergence towards the lower branch of the critical manifold before switching to a periodic orbit. }
	\label{fig:ChaosDifferentTrajectories}
\end{figure}

In Figure~\ref{fig:ChaosDifferentTrajectories}, we represent different branches of the trajectories on the phase plane $(x,y)$, together with the bifurcation diagram of the fast system. In this diagram, we represented the regions of $y$ for which the fast system has a topology similar to that represented in Figure~\ref{fig:PhaseSpaceChaos}, i.e. the presence of a stable periodic orbit and a stable fixed point, two unstable fixed point and no unstable periodic orbit. This region overlaps with the Hopf bifurcation point, and makes the system very sensitive to the specific shape of the trajectory. We represented on this diagram an actual trajectory of the system during one period of the relaxation oscillation, in all three cases (MMO, burst and mixed trajectory).

\section{Discussion}
We have considered here a simple model of delay differential equation based on the FitzHugh-Nagumo neuron model, that provides a  canonical example of canard explosion. This system shows a surprisingly rich phenomenology. For small delays, a classical canard explosion occurs similarly to what happens in the non-delayed FitzHugh-Nagumo equation, which is the topic of~\cite{krupa-touboul:14a}. However, as delays are increased, they destabilize branches of the critical manifold that are necessary to support canard cycles, and complex dynamics occur. As delays are increased, the first phenomenon observed is the emergence of instabilities of the small cycles. As shown in~\cite{krupa-touboul:14a}, these are related to an increase of the expansion rate along the critical manifold. Indeed, although the stability of the branches of the critical manifold remains unchanged, increasing the delay induces a decrease of the contraction rate of the trajectories along the stable branch of the critical manifold passing below the increasing expansion rate along the unstable manifold. This induces a sequence of bifurcations of canards cycles, and numerical evidence {shows} the emergence of period doubling bifurcations and chaotic small oscillations. As delays are further increased, stable branches of the critical manifold may lose stability. In that case, there is no more room for canard cycles. In the case of the FhN system, the branches of the critical manifold lose stability through a subcritical Hopf bifurcation emerging from one of the folds through a Bogdanov-Takens bifurcations. This results, {in particular,  in the emergence of a new type of MMOs}. For even larger delays, the fast system shows the presence of a stable limit cycle, which yields fast oscillations and bursting in the full system. Chaotic trajectories showing irregular alternations of MMO- and burst-like trajectories were also found within a specific parameter region, in which stable manifolds in the extended phase plane $(x_{t-\tau},x_{t})$ show a very specific shape. 

Delay differential equations are infinite dimensional dynamical systems, and as such may have very complex dynamics. However, we conjecture that the phenomena observed in this system may be observed in a three dimensional system with two fast and one slow variables. Such a toy model would allow going deeper into a number of phenomena that we identified in our infinite-dimensional setting and that are still not completely understood from a theoretical viewpoint. These include accounting for chaotic small oscillations when expansion overcomes contraction, the development of a single Hopf bifurcation curve with a Bautin point out of the merging of singular Hopf bifurcations and Bogdanov-Takens bifurcation in the fast system. Eventually, such a finite-dimensional version of the system would allow understanding more precisely the nature of the MMOs emerging here in the absence of two slow variables, as well as the chaotic orbits mixing MMOs- and burst-like trajectories. 

From the application viewpoint, the model analyzed is a simple representation of the activity of neuronal networks. This relative simplicity allowed us to uncover a number of interesting phenomena that are indeed observed in collective firings of cells, such as transitions from subthreshold oscillations to MMOs and bursting, that are for instance observed in the inferior olive~\cite{bernardo-foster:86}. We expect that the phenomena observed here persist in more complex systems and with different kinds of synaptic interactions, as long as these have an effective negative delayed feedback and the same kind of local bifurcation as the delayed FhN system. Our study was based on assuming in particular $\gamma=0$ in equation~\eqref{eq:Delayed}. Simulations of the delayed FhN model with $\gamma>0$ are proposed in appendix~\ref{append:FhN} and indeed show that one recovers the same phenomenology. Such phenomena were never attributed to the presence of delays. The observations made in the present manuscript now indicate that bursting, spiking and subthreshold oscillations may arise from the same model but with different effective delays in the communications between cells, and could account for the different spike patterns observed in different cortical areas.

\appendix

\section{Normal form reduction at Hopf bifurcation}\label{append:NormalForms}
In sections~\ref{sec-Hopfexist} and~\ref{sec:fastHopf}, we have identified parameters for which the linearized equation at a fixed point, respectively of the full system~\eqref{eq:FhNdelay} or of its fast subsystem, have a pair of purely imaginary eigenvalues. This provides us with the locus of possible Hopf bifurcations, and we show here aided by numerical evaluations that the system undergoes generic Hopf bifurcations. To this end one needs to reduce the system to normal form on the center manifold. This appendix is devoted to providing the outline of the method, in the context of the full system (slightly more complex than the fast one), and along the branch of Hopf bifurcations given by the first branch $\tau=\tau_1^0(J,\eps,a)$ for $a\in(1,\sqrt{1+2J})$, defined in equation~\eqref{eq:taustar}. This curve correspond to a change of stability of the fixed point $(a,a-a^3/3)$. We showed that as $\eps\to 0$, the curve collapses to a non-smooth curve made of a segment $\{J\tau <1, a=1\}$ and the curve of Hopf bifurcations of the fast system $\tau(a)$ given by equation~\eqref{eq:taufast}. We further showed that for when $\eps\to 0$, the Hopf bifurcation corresponding to $J\tau<1$ is supercritical, and otherwise subcritical.

We confirm this change of criticality of the Hopf bifurcation for $\eps>0$ by reducing the system to normal form reduction along the Hopf bifurcation manifold. We follow a relatively standard tradition exposed for instance for delayed differential equations in~\cite{hale-lunel:93}, and reviewed in~\cite{campbell2009calculating}. We outline the method and calculations performed aided by the formal calculations software Maple\textregistered, before presenting the obtained coefficients and their shape as a function of the parameters.

This method is based on considering the delay differential equation~\eqref{eq:FhNdelay} as a dynamical system in the Banach space $X$ of continuous functions from $[-\tau,0]$ to $\R^2$ endowed with the uniform norm. 
\[\Vert z\Vert = \sup_{\theta\in [-\tau,0]} \vert z(t)\vert\]
where $\vert z(t)\vert$ is the Euclidian norm on $\R^2$ of the value of the function $z\in X$ at time $t\in [-\tau,0]$. The delayed differential equation is expressed as a functional differential equation on this space. Indeed, defining now $z_t \in X$ as the portion of solution $(x(t),y(t),t\in [t-\tau,t])$, with the definition
\[z_t(\theta)=z(t+\theta), \quad -\tau\leq \theta \leq 0,\]
we can rewrite equation~\eqref{eq:FhNdelay} as:
\[\frac{d}{dt} x_t(\theta)=\begin{cases}
	\frac{d}{d\theta} x_t(\theta) & -\tau \leq \theta <0\\
	\Big[y_t(0)+ x_t(0)+ J(x_t(0)-x_t(-\tau))\Big] -\frac{x_t(0)^3}{3} & \theta=0
\end{cases}\]
and similarly, 
\[\frac{d}{dt} y_t(\theta)=\begin{cases}
	\frac{d}{d\theta} y_t(\theta) & -\tau \leq \theta <0\\
	\Big[\eps (a-x_t(0))\Big] & \theta=0
\end{cases}\]
The terms within the brackets is the affine part of the flow, separated from the nonlinear cubic term in the equation on the first variable, which does not involve delays. Calculations are much simplified when changing variables so that the equilibrium is at the origin. This change of variable simply amounts changing the origin, i.e. consider the equation in terms of $\tilde{x}=x-x^*$. The equations now read:
\begin{align*}
	\frac{d}{dt} \tilde{x}_t(\theta)=&\begin{cases}
		\frac{d}{d\theta} \tilde{x}_t(\theta) & -\tau \leq \theta <0\\
		\Big[(1-a^2)\tilde{x}_t(0)+ J(\tilde{x}_t(0)-\tilde{x}_t(-\tau))\Big] - a\tilde{x}_t(0)^2 - \frac{x_t(0)^3}{3} & \theta=0
	\end{cases},\\
	\frac{d}{dt} \tilde{y}_t(\theta)=&\begin{cases}
		\frac{d}{d\theta} \tilde{y}_t(\theta) & -\tau \leq \theta <0\\
		-\eps \tilde{x}_t(0) & \theta=0
	\end{cases}.
\end{align*}
The linear operator decomposes into a part only depending on $\tilde{z}_t(0)$:
\[A_0=\Matrix{1-a^2+J}{1}{-\eps}{0}, \]
and an operator depending on $\tilde{x}_t(-\tau)$:
\[A_1=\Matrix{-J}{0}{0}{0}.\] 
At Hopf bifurcation points, i.e. when the characteristic equation
\[det(\lambda I_2 -A_0 - e^{-\lambda \tau}A_1)=0\]
has a complex solution $\lambda=\pm \mathbf{i}\omega$, a complex right eigenvector is:
\[v=\Vector{1}{\frac{\mathbf{i}\eps}{\omega}}, \]
which corresponds to a two-dimensional center eigenspace $N$:
\[\Matrix{\cos(\omega\theta)}{\sin(\omega\theta)}{\displaystyle{-\frac{\sin(\omega\theta)\eps}{\omega}}}{\displaystyle{\frac{\cos(\omega\theta)\eps}{\omega}}}\]
and an infinite-dimensional stable eigenspace $S$. The corresponding center manifold is given by:
\[\mathcal{M} = \{\phi \in X, \; \phi = \Phi u + h(u)\}\]
where $u=(u_1,u_2)^t$ are the coordinates on the nullspace $N$ and $h(u)\in S$. Classically, we project the solutions to the delay differential equation on $\mathcal{M}$ and obtain:
\begin{multline*}
	z_t(\sigma)=\Vector{\cos(\omega\sigma)u_1(t) + \sin(\omega\sigma)u_2(t) }{\displaystyle{-\frac{\sin(\omega\theta)\eps}{\omega}u_1(t)+\frac{\cos(\omega\theta)\eps}{\omega}u_2(t)}} \\
	+ \Vector{\displaystyle{h_{11}^1 (\sigma) u_1(t)^2+h_{12}^1 (\sigma) u_1(t)u_2(t)+h_{22}^1 (\sigma) u_2(t)^2}}{\displaystyle{h_{11}^2 (\sigma) u_1(t)^2+h_{12}^2 (\sigma) u_1(t)u_2(t)+h_{22}^2 (\sigma) u_2(t)^2}} + \mathcal O (\Vert x\Vert^3)
\end{multline*}
where the $h_{jk}^i$ characterize the Taylor expansion of the solution on the center manifold. These coefficients satisfy a system of affine ODEs, in which the affine term depends on the orthogonal basis of the nullspace, denoted $\psi(\theta)$ (which is straightforward to calculate) arising from the projection of the original equation on $N$. These are now classical methods, introduced and used in~\cite{campbell2009calculating,stone2004stability,faria1995normal}. In our case, the solution to the linear ordinary differential equation of $h_{jk}^i$ yield relatively complex expressions, in terms of six constants $C_1\cdots C_6$, that are then solved in order to match boundary values (and solve the original DDE on the center manifold). One finds:
\[h_{11}^1 = \frac a {3\omega} \big(\cos(\omega\theta)\Psi_{21}-\sin(\omega\theta)\Psi_{11}(0)\big) + C_2 - C_5 \sin(\omega\theta)\cos(\omega\theta) + C_6  (\cos(\omega\theta)^2-\frac 1 2)\]
and similar expressions for the other terms, and it is not hard to determine the constants. From these expressions, we obtain a system of ODEs describing the evolution in time of $u$:
\[\begin{cases}
	\dot{u_1}=\omega u_2 + \Psi_{12}(0) (G_1(u_1,u_2) ) + \mathcal O (\Vert u\Vert^4)\\
	\dot{u_2}=\omega u_1 + \Psi_{22}(0) (G_1(u_1,u_2) ) + \mathcal O (\Vert u\Vert^4)
\end{cases}\]
with $G$ a cubic polynomial, whose coefficients can be deduced from the evaluation of $h_{jk}^i$ and $C_i$, and the first Lyapunov coefficient, characterizing the type of the Hopf bifurcation, is a simple function of these coefficients. With the help of the formal calculations software Maple\textregistered, we compute all these coefficients analytically. It happens that, as of quadratic terms, only the coefficients of $u_1^2$ in $G_1$ and $G_2$ are non zero, and cubic coefficients involved in the computation of the Lyapunov coefficients enjoy a relatively simple form. These coefficients are given by:
\[\begin{cases}
	G_1(u_1,u_2)=-a \Psi_{11}(0) u_1^2 + u_1^3 \Psi_{11}(0)\left(-\frac 1 3 -2 h^1_{11}(0)a\right)- 2\Psi_{11}(0) a h_{22}^1(0) + \cdots\\
	G_1(u_1,u_2)=-a \Psi_{21}(0) u_1^2 -2 a h_{12}^1(0) u_1^2 u_2 + \cdots\\
\end{cases}\]
where the dots correspond to terms that are not necessary to compute the first Lypaunov coefficient. From this expression, classical formulae (see e.g.~\cite{kuznetsov}) yield a very complex expression for the first Lypaunov coefficient as a function of the parameters. This expression is however exact, and allows direct numerical computation of the Lyapunov coefficient, that we presented in section~\ref{sec:Lyapu}.

\section{The FitzHugh-Nagumo system with $\gamma\neq 0$}\label{append:FhN}
In this appendix, we provide one numerical example of delayed self-coupled FitzHugh Nagumo system~\eqref{eq:Delayed} with $\gamma\neq 0$ that has similar dynamics as the ones studied in the main text, i.e. for $\gamma$ small enough. 
\begin{equation}\label{eq:DelayedFhN}
	\begin{cases}
		\dot{x}=x-\frac{x^3}{3}+y + J(x(t)-x(t-\tau))\\
		\dot{y}=\eps(a-x+\gamma y)
	\end{cases}
\end{equation}
The fixed points of the system are given by $y=(x-a)/\gamma$, where $x$ is the solution of the cubic polynomial:
\[\frac{\gamma+1}{\gamma} x -\frac{x^3}{3} -\frac a \gamma = 0.\]
This equation can be solved using Cardano method, exactly as we solve the cubic equation of the fast system. In detail, the Cardano discriminant of the equation writes:
\[\Delta = \frac{27}{b^2} (9 a^2 -\frac{b+1}{b}).\]
When $\Delta>0$, the system has a unique solution:
\[x_0=\left(-\frac{3a}{2b} + \sqrt{\Delta}\right)^{1/3} - \left(\frac{3a}{2b} + \sqrt{\Delta}\right)^{1/3},\]
and when $\Delta<0$, we find 3 roots:
\[x_k = 2\left(\frac{b+1}{b}\right)^{1/3} \cos\left(\frac 1 3  \arccos\left( -\frac{3a\sqrt{b}}{2 \sqrt{(b+1)^3}}\right) + \frac{2 k \pi}{3}\right).\]
The stability of the fixed points can be analyzed in a similar fashion as done before. In that case, denoting by $x_*$ one of the equilibria found, the dispersion relationship corresponds to the determinant of the matrix:
\[\xi \Matrix {1}{0}{0}{1}- \Matrix{A}{1}{-\eps}{\eps \gamma}\]
with $A=1-(x_*)^2 +J (1-e^{-\xi \tau})$, and the dispersion relationship therefore reads:
\[\xi^2- (A+\gamma \eps)\xi+\eps=0,\]
or:
\[e^{-\xi\tau}=\frac{\xi^2-\xi(1-(x_*)^2 +J+\gamma \eps)+\eps (\gamma(1-(x_*)^2 +J)+1)}{J\xi +\gamma \eps J}.\]
This appears much more complex to solve in closed form. 

The fast dynamics is identical as that of the delayed FhN system, and therefore the analysis of section~\ref{sec:fastHopf} applies. This allows to account for most of the behaviors of the FhN system. The only phenomenon requiring some more work is the presence of canard explosion and their type. Similarly to the non-delayed case, the details of the slow dynamics matter for the demonstration of the presence of canard explosions. Here, we provide numerical evidences for canard explosion, and show that similar phenomena of emergence of Mixed-Mode oscillations and bursting appear, for the same values of parameters, in the parameter regime where FhN system has a unique fixed point. Let us for instance consider the FhN model for $b=-1$ and $\gamma=-0.3$. We set $a=0.85$ (which leads the system to a similar situation as studied in the FhN system with $a=1.01$). For $\tau>1/J$, we have seen that geometric analysis of the fast system accounted for most phenomena arising in the FhN system. This property persists for the FitzHugh-Nagumo system, and we therefore observe the very same behaviors and transitions at similar delays. 

\begin{figure}[htbp]
	\centering
		\includegraphics[width=.8\textwidth]{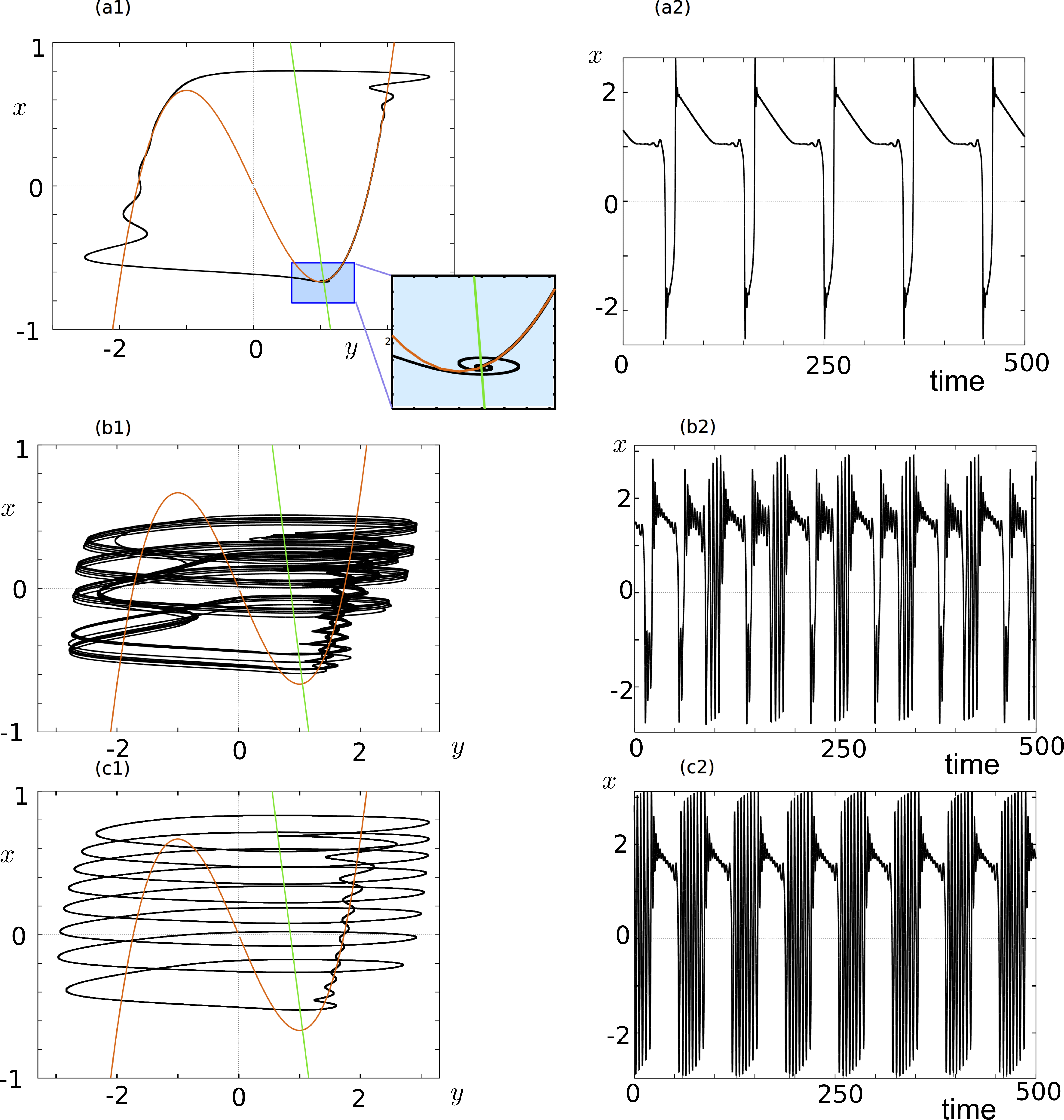}
	\caption{Self-coupled delayed FitzHugh-Nagumo system for $\tau>1/J$ and $\eps=0.05$. (a) $\tau=0.6$, (b) $\tau=0.7$, (c)$\tau=0.9$. The system transitions from MMOs to bursting through a period of chaos, arising for similar parameter values as in the delayed FhN system. }
	\label{fig:label}
\end{figure}
%

\section*{Conflict of Interest:} 
The authors declare that they have no conflict of interest and that they have complied with ethical standards.

\end{document}